\documentclass[a4paper,reqno]{amsart}

\calclayout

\usepackage{fancyhdr}
\usepackage{csquotes}
\usepackage{mathtools}
\usepackage{amsmath}
\usepackage{amssymb}
\usepackage{wasysym}
\usepackage{stmaryrd}
\usepackage{amsthm}
\usepackage{enumitem}
\usepackage{graphicx}
\usepackage{tikz-cd}
\usepackage{wrapfig}
\usepackage{DotArrow}
\usepackage[%
backend=bibtex,%
style=alphabetic,%
natbib=false,%
isbn=false, 
doi=false, 
url=false, 
eprint=true, 
giveninits=true, 
backref=false, 
]{biblatex} \AtEveryBibitem{%
	\ifentrytype{unpublished}{}{}%
} \renewbibmacro{in:}{} 
\usepackage{doi}%
\usepackage{xurl}
\usepackage{hyperref}
\hypersetup
{
	breaklinks = true,
	colorlinks,
	linkcolor = {blue},
	citecolor = {red},
	urlcolor  = {red},
}

\usepackage{adjustbox}

\DeclareMathOperator{\Hom}{Hom}
\DeclareMathOperator{\Ext}{Ext}
\DeclareMathOperator{\id}{id}
\DeclareMathOperator{\Mod}{mod}

\DeclareMathOperator{\Ker}{Ker}
\DeclareMathOperator{\Ima}{Im}
\DeclareMathOperator{\Coker}{Coker}
\DeclareMathOperator{\thick}{thick} 
\DeclareMathOperator{\per}{per} 
\DeclareMathOperator{\cone}{cone}
\DeclareMathOperator{\Cocone}{cocone}
\DeclareMathOperator{\add}{add}

\DeclareMathOperator{\rad}{rad}
\DeclareMathOperator{\Top}{top}
\DeclareMathOperator{\soc}{soc}
\DeclareMathOperator{\Irr}{Irr}
\newcommand{\Dual}{\mathbb{D}}
\newcommand{\proj}{\Cat{P}}
\newcommand{\inj}{\Cat{I}}
\newcommand{\Dproj}{\Cat{D}\mathrm{Proj}(A)}
\newcommand{\Dinj}{\Cat{D}\mathrm{Inj}(A)}
\newcommand{\Cat}[1]{\mathcal{#1}}
\newcommand{\D}{\Cat{D}}
\newcommand{\T}{\Cat{T}}

\newcommand{\C}{\Cat{C}}

\newcommand{\Hc}[1]{\Cat{H}^{{#1}}_A}

\newcommand{\Double}[1]{\mathbb{#1}}
\newcommand{\EE}{\Double{E}}
\newcommand{\DD}{\Double{D}}

\newcommand{\fd}{\D^{\mathrm{fd}}}

\newcommand{\definef}[1]{\emph{{#1}}}
\newcommand{\define}[1]{{#1}}
\newcommand{\trunc}{t}
\newcommand{\Proj}{\mathcal{P}}
\newcommand{\rk}{\mathrm{rk}}

\newcommand\Ar[3]{\ar[from={#1}, to={#2}, #3]}
\newcommand\Arop[3]{\ar[from={#2}, to={#1}, #3]}
\newcommand\tauAr[3]{\ar[from={#1}, to={#2}, #3,color=darkgray, dash pattern=on 3pt off 4pt]}
\newcommand\red[1]{\textcolor{red}{#1}}
\newcommand{\blue}[1]{\textcolor{blue}{#1}}
\newenvironment{xsmallmatrix}
  {\renewcommand\thickspace{\kern.0em}\smallmatrix}
  {\endsmallmatrix}

\newcommand{\leaveout}[1]{}

\usepackage[capitalise, noabbrev]{cleveref}
\theoremstyle{plain}  
\newtheorem{theorem}{Theorem}[section]
\newtheorem{prop}[theorem]{Proposition}
\newtheorem{lemma}[theorem]{Lemma}
\newtheorem{coro}[theorem]{Corollary} 
\theoremstyle{definition}
\newtheorem{definition}[theorem]{Definition}
\newtheorem{rema}[theorem]{Remark}
\newtheorem{ex}[theorem]{Example}

\title[On the Auslander--Reiten Theory for Extended Hearts]{On the Auslander--Reiten Theory for Extended Hearts of Proper Connective DG Algebras}

\author[N.~Mochizuki]{Nao Mochizuki}%
\address{%
	Graduate School of Mathematics, %
    Nagoya university, 
    Furocho, %
    Chikusaku, %
    Nagoya, %
    Japan
}%
\email{mochizuki.nao.n8@s.mail.nagoya-u.ac.jp}%

\author[M.~Plogmann]{Marvin Plogmann}%
\address{%
	Mathematisches Institut, %
	Universität zu Köln, %
	Weyertal 86-90, %
	50931 Köln, %
	Germany%
}%
\email{plogmann@math.uni-koeln.de}%
\urladdr{https://sites.google.com/view/marvinplogmann}%



\begin{document}

    \begin{abstract}
		We prove that, for a proper connective dg algebra $A$ with cohomology concentrated in degrees between $1-d$ and $0$, the extended heart $\fd(A)^{(-d,0]}\subseteq \fd(A)$ is an extriangulated category with almost-split conflations. We also prove a version of the 1st Brauer--Thrall Conjecture in this context.
    \end{abstract}
	
    \maketitle
	
    \section{Introduction}
    Let $k$ be a field and $A$ a connective dg algebra over $k$, that is $H^{*}(A)$ is concentrated in non-positive cohomological degrees. We denote by $\D(A)$ the (unbounded) derived category of (right) dg $A$-modules, and by $\fd(A)\subseteq \D(A)$ the full subcategory of dg $A$-modules $M$ whose cohomology $H^{*}(M)$ is finite dimensional. In the following, we fix $A$ to be \definef{proper}, \emph{i.e.} $H^{*}(A)$ is finite dimensional. We will use the notation $\Hom_A:=\Hom_{\D(A)}$ in what follows.
	
    The standard $t$-structure $(\D(A)^{\leq 0},\D(A)^{\geq 0})$ on $\D(A)$ is given by the full subcategories
    \[\D(A)^{\leq 0}\coloneqq \{M\colon H^{i}(M)=0 \text{ for all }i>0\}, \quad \D(A)^{\geq 0}\coloneqq \{M\colon H^{i}(M)=0 \text{ for all }i<0\};\]
    it restricts to a $t$-structure on $\fd(A)$ whose heart $\mathcal{H}$ is the abelian category of finite-dimensional modules over the finite-dimensional algebra $H^0(A)$, denoted $\Mod(H^0(A))$. For $n\in\mathbb{Z}$ we define $\fd(A)^{\leq n}\coloneqq \fd(A)^{\leq 0}[-n]$, $\fd(A)^{\geq n}\coloneqq \fd(A)^{\geq 0}[-n]$ and $\fd(A)^{< n}\coloneqq \fd(A)^{\leq n-1}$, $\fd(A)^{> n}\coloneqq \fd(A)^{\geq n+1}$. The main object of study in this article is the \definef{extended heart} 
    \[\define{\Hc{d}}\coloneqq \fd(A)^{\leq 0}\cap \fd(A)^{>-d}\] 
    for an integer $d\geq 1$. As both $\fd(A)^{\leq 0}$ and $\fd(A)^{>-d}$ are extension closed, $\Hc{d}\subseteq \fd(A)$ is extension closed and hence all of them are extriangulated by \cite[Rem.~2.18]{NP19}. Recall, for an extension closed subcategory $\mathcal{S}$ of a triangulated category $\T$, the biadditive functor $\EE:{\mathcal{S}}^{op}\times\mathcal{S}\to \Mod (k)$ is given by $\EE(-,?)=\Hom_{\T}(-,?[1])$ and the class of \definef{conflations} in $\mathcal{S}$ consists of the triangles 
    \[L\overset{i}{\longrightarrow}M \overset{p}{\longrightarrow}N{\longrightarrow}L[1]\]
    in $\T$ such that $L,M,N\in\mathcal{S}$. In this case $i$ is called an \definef{inflation} and $p$ is called a \definef{deflation}.
    
    \begin{rema}\label{RemarkExHeart}
		It follows from the definitions, that
		\begin{enumerate}
			\item $f\in\Hc{d}(X,Y)$ is an inflation if and only if $H^{-d+1}(f)$ is a monomorphism in $\mathcal{H}$,
			\item $f\in\Hc{d}(X,Y)$ is a deflation if and only if $H^{0}(f)$ is an epimorphism in $\mathcal{H}$.
		\end{enumerate}
		For a proof see for example \cite[Thm.~4.50]{M25}. Moreover, by \cite[Thm.~4.44]{M25} the canonical dg enhancement of $\Hc{d}$ is an abelian $(d,1)$-category in the sense of \cite[Def.~4.12]{M25}; we will, however, make no use of this latter fact.
    \end{rema}
    
    Recall that an object $P$ of an extriangulated category $(\C,\EE,\mathfrak{s})$ is called \definef{projective} if $\EE(P,X)=0$ for all $X\in\C$. An object $I$ of $\C$ is called \definef{injective} if $\EE(X,I)=0$ for all $X\in\C$. Denote by $\define{\proj(\C)}$ and $\define{\inj(\C)}$, the full subcategories of projective and injective objects in $\C$ respectively. We say that $\C$ has \definef{enough projectives}, if for all $X\in\C$ there exists $P\in\proj(\C)$ and a deflation $p\colon P\to X$. Dually, $\C$ has \definef{enough injectives}, if for all $X\in\C$ there exists $I\in\inj(\C)$ and an inflation $\iota\colon X\to I$. In this case we denote by $\underline{\C}\coloneqq \C/\proj(\C)$ and $\overline{\C}\coloneqq \C/\inj(\C)$ the stable and costable category respectively (see \cite[Def.~1.21,Rem.~1.22]{INP24}).
	
    Since $A$ is proper, there exists an integer $d\geq 1$ such that $A\in\Hc{d}$; and we fix such $d$ for the rest of the note. In this case it follows from \cite[Thm.~3.15]{Bon24}, that $\Hc{d}$ has enough projectives given by $\proj(\Hc{d})=\add(A)$. By dualising the arguments in \cite{Bon24} we will see in \cref{EnoughPI} that $\Hc{d}$ has enough injectives, which are given by $\inj(\Hc{d})=\add(\DD A[d-1])$, where  
    \[\define{\DD(-)}\coloneqq\mathbb{R}\Hom_k(-,k)\colon \fd(A^{op})^{op}\to \fd(A)\]  
    denotes the \definef{$k$-linear duality functor}. 
	
    In the case $A\simeq H^0(A)=:\Lambda$, the extended heart $\mathcal{H}_{\Lambda}^d$ was studied in \cite{G24} and \cite{Z25}. In \cite[Th.~4.1]{G24}, a bijection between torsion pairs in $\mathcal{H}_{\Lambda}^d$, cotorsion pairs in $\mathrm{K}^{[-d,0]}(\mathrm{proj}(\Lambda))$ and $(d+1)$-term silting complexes is proven. In \cite[Th.~4.7]{Z25} these bijections are extended to include basic $\tau$-tilting pairs in $\mathcal{H}_{\Lambda}^d$. In particular, it is proven in \cite[Th.~3.12]{Z25} that $\mathcal{H}_{\Lambda}^d$ has an Auslander--Reiten--Serre duality and hence almost-split conflations in the sense of \cite[Def.~2.7]{INP24}, which allow to give a description of (part of) the Auslander-Reiten quiver of $\mathcal{H}_{\Lambda}^d$. We prove the following generalization of \cite[Thm.~3.12]{Z25} in the setting of proper connective dg algebras, which was known to be true for $d$-self-injective dg algebras by \cite[Thm.~3.1]{j20}.
	
    \begin{theorem}\label{MainTheorem}
		Let $A$ be a proper, connective dg algebra and $d>0$ such that $A\in \Hc{d}$. Then, there exists an \definef{Auslander--Reiten--Serre duality} in the sense of \cite[Def.~3.4.]{INP24}, that is a pair $(\tau,\eta)$ of a $k$-linear equivalence $\tau\colon\underline{\Hc{d}}\to \overline{\Hc{d}}$ and a binatural isomorphism 
		\[\eta_{X,Y}\colon \underline{\Hc{d}}(X,Y) \cong \DD\Hom_{\fd(A)}(Y,\tau X[1])  \quad \text{for any }X,Y\in{\Hc{d}}.\]
    \end{theorem}
	
    \begin{coro}\label{AlmostSplitCoro}
		Let $A$ and $d$ as above. Then, the extended heart $\Hc{d}$ is an extriangulated category with almost-split conflations.
    \end{coro}
	
    \begin{proof}
		By \cite[Prop.~6.12]{AMY19}, $\Hc{d}$ is a $k$-linear $\Ext$-finite extriangulated category. Moreover, $\D(A)$ is idempotent complete by combining \cite[Sec.~4.1]{Kel94} and \cite[Prop.~1.6.8]{Nee01}. Now $\Hc{d}$ is closed under direct summands in $\D(A)$ and hence $\Hc{d}$ is Krull-Schmidt by \cite[Coro.~4.4]{Kra15}. Therefore, $\Hc{d}$ has almost-split conflations by \cite[Thm.~3.6]{INP24} and \cref{MainTheorem}.
    \end{proof}

    We prove the following version of the 1st Brauer--Thrall Conjecture in this context by generalising the Harada--Sai Lemma for $\Hc{d}$. We refer the reader to \cref{AR-quivers} for the definitions.

    \newtheorem*{theoremRepeat}{Theorem {\ref*{BT1}}} 
    \begin{theoremRepeat}
        Let $A$ be $d$-connected and $\C$ a bounded connected component of the Auslander--Reiten quiver of $\Hc{d}$. Then $\C$ is finite and $\C$ is the whole of the Auslander--Reiten quiver of $\Hc{d}$.
    \end{theoremRepeat}
	
    The main part of this note consists in constructing the additive equivalence $\tau$ in the statement above. Our construction is a generalisation of the following construction of the Auslander--Reiten translation for a finite dimensional algebra $\Lambda$. In fact, for the case $A\simeq H^0(A)$ and $d=1$, our construction is exactly the same and, still if $A\simeq H^0(A)$ but $d\geq 1$, our construction agrees with \cite[Def.~3.7]{Z25}.
	
    Recall the Nakayama functor $\nu=-\otimes_{\Lambda} \DD \Lambda$, which induces an equivalence 
    \[[-1]\circ \nu:\mathrm{K}^{[-1,0]}(\mathrm{proj}(A))\stackrel{\sim}{\longrightarrow} \mathrm{K}^{[0,1]}(\mathrm{inj}(A)).\] 
    Then $\tau$ is defined by the following commutative diagram:
	
    \[\begin{tikzcd}
		{\frac{\mathrm{K}^{[-1,0]}(\mathrm{proj(\Lambda)})}{\add(\Lambda[1]\oplus \Lambda)}} & {\frac{\mathrm{K}^{[0,1]}(\mathrm{inj(\Lambda)})}{\add(\DD \Lambda[-1]\oplus \DD \Lambda)}} \\
		{\underline{\Mod}(\Lambda)} & {\overline{\Mod}(\Lambda)}
		\arrow["{[-1]\circ \nu}", from=1-1, to=1-2]
		\arrow["\sim"', draw=none, from=1-1, to=1-2]
		\arrow["{\mathrm{coker}\cong H^0}"', from=1-1, to=2-1]
		\arrow["\sim", draw=none, from=1-1, to=2-1]
		\arrow["{H^{0}\cong \mathrm{ker}}", from=1-2, to=2-2]
		\arrow["\sim"', draw=none, from=1-2, to=2-2]
		\arrow["\tau", from=2-1, to=2-2]
    \end{tikzcd}\]

    The \definef{derived Nakayama functor} is defined as
    \[\define{\nu}=\define{\nu_A}\coloneqq-\otimes^{\mathbb{L}}_A\DD(A)\colon\D(A)\to\D(A),\] 
    and restricts to a triangle equivalence $\per(A):=\thick(A)\stackrel{\sim}{\to}\thick(\DD A)$, see {\cite[Sec.~4.5]{KY14}}. Moreover, for $P\in\per(A)$ and $M\in\fd(A)$, we have the relative Serre-duality formula
    \begin{equation}\label{AR-formula}
		\Hom_{A}(P,M)\cong \DD \Hom_{A}(M,\nu_AP),
    \end{equation}
    which is natural in $P$ and $M$, see \cite[Sec.~10.4]{Kel94}.
	
    For $n\in\mathbb{Z}$, recall the inclusion functors, as well as their right and left adjoint \definef{truncation functors} respectively: 
    \[\begin{tikzcd}
		\fd(A)^{\leq n} & \fd(A) & \fd(A)^{\geq n}
		\arrow[hook,""{name=0, anchor=center, inner sep=0}, "{\iota^{\leq n}}", shift left=2, from=1-1, to=1-2]
		\arrow[""{name=1, anchor=center, inner sep=0}, "\trunc^{\leq n}", shift left=2, from=1-2, to=1-1]
		\arrow[""{name=2, anchor=center, inner sep=0}, "\trunc^{\geq n}", shift left=2, from=1-2, to=1-3]
		\arrow[hook',""{name=3, anchor=center, inner sep=0}, "{\iota^{\geq n}}", shift left=2, from=1-3, to=1-2]
		\arrow["\dashv"{anchor=center, rotate=-90}, draw=none, from=0, to=1]
		\arrow["\dashv"{anchor=center, rotate=-90}, draw=none, from=2, to=3]
    \end{tikzcd}\]
	
    Recall the following notation from \cite[Sec.~2]{IY08}. For a triangulated category $\T$ and two full subcategories $\mathcal{X},\mathcal{Y}\subseteq\T$, we denote by $\mathcal{X}\ast\mathcal{Y}$, the full subcategory of $\T$ consisting of objects $T$ such that there exist $X\in\mathcal{X}, Y\in\mathcal{Y}$ and a triangle
    \[X\overset{a}\longrightarrow T \overset{b}\longrightarrow Y\overset{c}\longrightarrow X[1].\]
    The octahedral axiom yields the associativity of the binary operation $(\mathcal{X},\mathcal{Y})\mapsto \mathcal{X}\ast\mathcal{Y}$. We denote by 
    \[\define{\per(A)^{[-d,0]}}\coloneqq\add(A)\ast\add(A)[1]\ast\cdots \ast\add(A)[d]\subseteq \per(A)\] 
    the full subcategory of \definef{$(d+1)$-term perfect dg $A$-modules}, that is the interval for the canonical co-$t$-structure on $\per(A)$. Dually, we define
    \[\define{\thick(\DD A)^{[0,d]}}=\add(\DD A)[-d]\ast\add(\DD A)[-d+1]\ast\cdots \ast\add(\DD A)\subseteq \thick(\DD A).\]
    
    \begin{rema}
		In the case $A\simeq H^0(A)=:\Lambda$ is a finite-dimensional algebra and $d=1$, the $2$-term perfect dg $A$-modules $\per(A)^{[-1,0]}\cong K^{[-1,0]}(\mathrm{proj}(\Lambda))$ is a motivating example of a $0$-Auslander extriangulated category considered in \cite{GNP23,Chen230,FGPPP24}. Hence, $\per(A)^{[-d,0]}$ could be regarded as a natural analogue for dg algebras. We make this analogy more precise in \cref{FinalSection}.
    \end{rema}
    
    As $\nu_A(A)\cong \DD(A)$ it follows that $\nu_A$ induces equivalences
    \begin{equation}\label{NakayamaEquiv}
		\begin{tikzcd}
			{\per(A)} & {\thick(\DD A)} \\
			{\per(A)^{[-d,0]}} & {\thick(\DD A)^{[0,d]}}
			\arrow["{\nu_A}", from=1-1, to=1-2]
			\arrow["\sim"', draw=none, from=1-1, to=1-2]
			\arrow[u, phantom, sloped, "\subseteq",from=2-1, to=1-1]
			\arrow["{[-d]\circ \nu_A}", from=2-1, to=2-2]
			\arrow["\sim"', draw=none, from=2-1, to=2-2]
			\arrow[u, phantom, sloped, "\subseteq",from=2-2, to=1-2]
		\end{tikzcd}
    \end{equation}
    
    Since $A\in\fd(A)$ it follows that $\per(A),\thick(\DD A)\subseteq \fd(A)$. In \cref{d-equiv-stable} and \cref{d-equiv-costable} we prove that $\trunc^{>-d}$ and $[d-1]\circ\trunc^{<d}$ induce equivalences
    \[\underline{\trunc^{>-d}}\colon \frac{\per(A)^{[-d,0]}}{\add\left(A[d]\oplus A\right)}\stackrel{\sim}\longrightarrow \underline{\Hc{d}}, \quad \quad
    \overline{[d-1]\circ\trunc^{<d}}\colon \frac{\thick(\DD A)^{[0,d]}}{\add\left(\DD A[-d]\oplus \DD A\right)}\stackrel{\sim}\longrightarrow \overline{\Hc{d}}.\]
	
    Finally, we denote by $(\underline{\trunc^{>-d}})^{-}$ a quasi-inverse of $\underline{\trunc^{>-d}}$ and define $\define{\tau}\colon\underline{\Hc{d}}\to \overline{\Hc{d}}$ by
    \begin{equation}\label{TauDef}
		\tau\coloneqq\overline{[d-1]\circ\trunc^{<d}}  \circ ([-d]\circ\nu_A) \circ (\underline{\trunc^{>-d}})^{-}, \quad 
		\begin{tikzcd}
			{\frac{\per(A)^{[-d,0]}}{\add\left(A[d]\oplus A\right)}} & {\frac{\thick(\DD A)^{[0,d]}}{\add\left(\DD A[-d]\oplus \DD A\right)}} \\
			{\underline{\Hc{d}}} & {\overline{\Hc{d}}}
			\arrow["{[-d]\circ\nu_A}", from=1-1, to=1-2]
			\arrow["\sim"', draw=none, from=1-1, to=1-2]
			\arrow["{(\underline{\trunc^{>-d}})^{-}}"', from=2-1, to=1-1]
			\arrow["\sim", draw=none, from=2-1, to=1-1]
			\arrow["{\overline{[d-1]\circ\trunc^{<d}}}", from=1-2, to=2-2]
			\arrow["\sim"', draw=none, from=1-2, to=2-2]
			\arrow["\tau", from=2-1, to=2-2]
		\end{tikzcd}
    \end{equation}

    \section{Proof of \cref{MainTheorem}}
    Following \cite{Bon24}, we denote by $\Dproj\coloneqq\proj(\fd(A)^{\leq 0})$ and $\Dinj\coloneqq\inj(\fd(A)^{\geq 0})$ the full subcategories of \definef{derived projective} and \definef{derived injective} objects respectively. Recall from \cite[Lem.~3.3]{Bon24} that $H^0(P)\in\mathrm{proj}(H^0(A))$. Analogously, one shows that $H^0(I)\in\mathrm{inj}(H^0(A))$.
	
    Let now $P\in\Dproj$ and $X\in\fd(A)^{\leq 0}$. We call a morphism $p\colon P\to X$ a \definef{derived projective cover} of $X$ if $H^0(p)$ is a projective cover in $\Mod(H^0(A))$. Dually, for $I\in\Dinj$ and $Y\in\fd(A)^{\geq 0}$, a morphism $i\colon Y\to I$ is called a \definef{derived injective envelope} of $Y$, if $H^0(i)$ is an injective envelope in $\Mod(H^0(A))$. By \cite[Thm.~3.15]{Bon24} there is an equality $\Dproj=\add(A)$ and every $X\in\fd(A)^{\leq 0}$ has a derived projective cover. 
	
    \begin{prop}\label{EnoughInj}
		Every object in $\fd(A)^{\geq 0}$ has a derived injective envelope.
    \end{prop}
	
    \begin{proof}
		We first observe that $\Dual(A)$ is derived injective, since for $X\in \fd(A)^{\geq 0}$
		\[\Hom_{A}(X,\DD(A)[1])\cong\Hom_A(X[-1],\nu_A(A))\stackrel{(\ref{AR-formula})}{\cong} \DD\Hom(A,X[-1]) \cong \DD H^{-1}(X)=0.\]
		Hence $\nu_A$ restricts to a functor
		\[\nu_A\colon\Dproj=\add(A)\to \add(\DD A)\subseteq\Dinj.\]
		
		Let $X\in\fd(A)^{\geq 0}$. Since $H^0(A)$ is a finite dimensional algebra, there exists an injective envelope $\iota^0\colon H^0(X)\to I^0$ in $\Mod(H^0(A))$. Moreover, there exists $P^0\in\mathrm{proj}(H^0(A))$ such that $\nu_{H^0(A)}(P^0)=I^0$. By \cite[Thm.~3.15]{Bon24}, there exists $P\in\Dproj$ such that $H^0(P)=P^0$. Define $I=\nu_A(P)\in\Dinj$. As $H^0(\nu_A(A))\cong H^0(\DD A)\cong \DD H^0(A)\cong \nu_{H^0(A)}(H^0(A))$ it follows 
		\[H^0(I)=H^0(\nu_A(P))\cong\nu_{H^0(A)}(H^0(P))=I^0.\] 
		Dually to \cite[Lem.~3.2]{Bon24} one proves that $H^0$ induces an isomorphism 
		\begin{equation}\label{H0Iso}
			H^0\colon\Hom_A(X,I)\overset{\cong}{\longrightarrow}\Hom_A(H^0(X),H^0(I)).
		\end{equation} 
		Let $\iota\colon X\to I$ be the image of $\iota^0$ under this isomorphism. This is by definition a derived injective envelope.
    \end{proof}
	
    \begin{coro}
		The functor $H^0\colon\fd(A)\to \Mod(H^0(A))$ restricts to an equivalence of categories 
		\[H^0\colon\Dinj\stackrel{\sim}\longrightarrow \mathrm{inj}(H^0(A)).\] In particular, $\Dinj=\add(\DD A)$
    \end{coro}
	
    \begin{proof}
        For every injective $H^0(A)$-module $I^0$ there exists a derived injective envelope $\iota\colon I^0\to I$ by \cref{EnoughInj}. As $H^0(\iota)$ is an injective envelope it holds $I^0\cong H^0(I)$. This proves density. Fully faithfulness follows from (\ref{H0Iso})
    \end{proof}
	
    \begin{coro}\label{EnoughPI}
		The extended heart $\Hc{d}$ has enough projectives and injectives as an extriangulated category. Moreover, $\proj(\Hc{d})=\add(A)$ and $\inj(\Hc{d})=\add(\DD(A)[d-1])$.
    \end{coro}
	
    \begin{proof}
		Let $X\in\Hc{d}$. By \cite[Thm.~3.15]{Bon24}, there exists $P\in\add(A)$ and $p\colon P\to X$ a derived projective cover. As $A\in\Hc{d}$ by assumption, it follows that $P\in\Hc{d}$ and $P$ is projective. Moreover $H^0(p)$ is an epimorphism by assumption, which means $p$ is a deflation by \cref{RemarkExHeart}. Hence, $\Hc{d}$ has enough projectives. In particular, for $P'\in\proj(\Hc{d})$, there is a deflation $P\to P'$ with $P\in\add(A)$. However, this has to split as $P'$ is projective, which means $P'\in\add(A)$.
		
		Dually, by \cref{EnoughInj} we obtain $I\in\add(\DD A)$ and $\iota\colon X[-d+1]\to I$ a derived injective envelope. Then it holds $I[d-1]\in\Hc{d}$ and for all $Y\in\Hc{d}$ we have 
		\[\EE(Y,I[d-1])\cong\Hom_A(Y,I[d])\cong\Hom_A(Y[-d],I)=0,\]
		since $Y[-d]\in\fd(A)^{>0}$. Therefore $I[d-1]\in\Hc{d}$ is injective. Finally, $H^{-d+1}(\iota[d-1])=H^0(\iota)$ is a monomorphism by assumption, and $\iota[d-1]\colon X\to I[d-1]$ is an inflation by \cref{RemarkExHeart}. Dually to the above we see $\inj(\Hc{d})=\add(\DD A[d-1])$.
    \end{proof}
	
    \begin{rema}\label{StableHom}
		Let $Y\in\fd(A)^{\leq 0}$ and $p\colon P\to Y$ be a derived projective cover of $Y$. Then, by \cite[Lem.~3.13]{Bon24}, $p$ is a minimal right approximation of $Y$ by derived projectives. Therefore, every morphism that factors through $\proj(\Hc{d})$, factors through $p$ and for every $X,Y\in\Hc{d}$, we obtain an exact sequence
		\[\Hc{d}(X,P)\overset{p_*}{\longrightarrow} \Hc{d}(X,Y)\to \underline{\Hc{d}}(X,Y)\to 0\]
    \end{rema}
	
    Analogously to the classical case we will prove the following result.
	
    \begin{theorem}\label{d-term-Equiv}
		The functor $\trunc^{>-d}\colon\per(A)^{[-d,0]}\to\Hc{d}$ induces an equivalence 
		\[\trunc^{>-d}\colon\frac{\per(A)^{[-d,0]}}{\add(A[d])}\stackrel{\sim}\longrightarrow \Hc{d}.\]
    \end{theorem}
	
    \begin{rema}
		\cref{d-term-Equiv} is a generalisation of \cite[Prop.~3.1]{G24}. Our proof differs from the one in \cite{G24} in the sense that we make use of the extriangulated structure on $\Hc{d}$.
    \end{rema}
	
    For the proof we need some preparation. First we observe that one can resolve objects in $\Hc{d}$ by $(d+1)$-term perfect dg $A$-modules. For $d=1$ this is simply the statement that every finite-dimensional $A$-module admits a projective presentation.
	
    \begin{lemma}\label{d-res-lemma}
		Let $n\in\mathbb{N}$ and $X\in\fd(A)^{\leq 0}$. Then there exists $P\in\per(A)^{[-n,0]}$ such that $\trunc^{>-n}P\cong \trunc^{>-n}X$. In particular, for $X\in\Hc{d}$ there exists $P\in\per(A)^{[-d,0]}$ such that $\trunc^{>-d}P\cong X$.
    \end{lemma}
	
    \begin{proof}
		There exists a derived projective cover $p^0\colon P^0\to X$ and hence $\Omega_A^1(X)\in\fd(A)^{\leq 0}$ such that there is a triangle
		\[\Omega_A^1(X)\overset{\iota^1}{\longrightarrow} P^0 \overset{p^0}{\longrightarrow} X \longrightarrow \Omega_A^1(X)[1]\]
		in $\fd(A)$. Inductively, we take derived projective covers $p^{-i}\colon P^{-i}\to \Omega_A^i(X)$ and objects ${\Omega_A^{i+1}(X)\in\fd(A)^{\leq 0}}$ such that there is a triangle
		\[\Omega_A^{i+1}(X)\overset{\iota^{i+1}}{\longrightarrow} P^{-i} \overset{p^{-i}}{\longrightarrow} \Omega_A^{i}(X) \longrightarrow \Omega_A^{i+1}(X)[1].\] 
		Define $P^{[-1,0]}\coloneqq\cone(\iota^1\circ p^{-1})\in\per(A)^{[-1,0]}$, and inductively by using the octahedral axiom $P^{[-(i+1),0]}\coloneqq\cone(g^{i}\circ p^{-(i+1)})\in \per(A)^{[-(i+1),0]}$.
		\[\adjustbox{scale=0.712}{
			\begin{tikzcd}
				{P^{-1}} & {\Omega_A^1(X)} & {\Omega_A^2(X)[1]} & {P^{-1}[1]}\\
				{P^{-1}} & {P^{0}} & {P^{[-1,0]}} & {P^{-1}[1]} \\
				& X & X  \\
				& {\Omega_A^1(X)[1]} & {\Omega_A^2(X)[2]}
				\arrow["{p^{-1}}", from=1-1, to=1-2]
                \arrow[from=1-2, to=1-3]
				\arrow[from=1-3, to=1-4]
                \arrow[equals,from=1-1, to=2-1]
				\arrow["{\iota^1}", from=1-2, to=2-2]
				\arrow["{g^1}", dashed, from=1-3, to=2-3]
                \arrow[from=1-4, to=2-4]
				\arrow[from=2-1, to=2-2]
				\arrow[from=2-2, to=2-3]
				\arrow[from=2-3, to=2-4]
				\arrow[from=2-2, to=3-2]
				\arrow["{f^1}", dashed, from=2-3, to=3-3]
                \arrow[equals, from=3-2, to=3-3]
				\arrow[from=3-2, to=4-2]
				\arrow[dashed, from=3-3, to=4-3]
                \arrow[from=4-2, to=4-3]
			\end{tikzcd}
		} \quad 
		\adjustbox{scale=0.712}{
			\begin{tikzcd}
				{P^{-(i+1)}[i]} & {\Omega_A^{i+1}(X)[i]} & {\Omega_A^{i+2}(X)[i+1]} & {P^{-(i+1)}[i+1]} \\
				{P^{-(i+1)}[i]} & {P^{[-i,0]}} & {P^{[-(i+1),0]}} & {P^{-(i+1)}[i+1]} \\
				& X & X \\
				& {\Omega_A^{i+1}(X)[i+1]} & {\Omega_A^{i+2}(X)[i+2]}
				\arrow["{p^{-(i+1)}}", from=1-1, to=1-2]
                \arrow[from=1-2, to=1-3]
				\arrow[from=1-3, to=1-4]
                \arrow[equals, from=1-1, to=2-1]
				\arrow["{g^i}", from=1-2, to=2-2]
				\arrow["{g^{i+1}}", dashed, from=1-3, to=2-3]
                \arrow[from=1-4, to=2-4]
				\arrow[from=2-1, to=2-2]
				\arrow[from=2-2, to=2-3]
				\arrow[from=2-3, to=2-4]
				\arrow["{f^i}", from=2-2, to=3-2]
				\arrow["{f^{i+1}}", dashed, from=2-3, to=3-3]
				\arrow[equals,from=3-2, to=3-3]
                \arrow[from=3-2, to=4-2]
				\arrow[dashed, from=3-3, to=4-3]
                \arrow[from=4-2, to=4-3]
			\end{tikzcd}
		}
		\]
		For $i=n$ we obtain a triangle
		\[\Omega_A^{n+1}(X)[n]\overset{g^{n}}{\longrightarrow} P^{[-n,0]} \overset{f^{n}}{\longrightarrow} X \longrightarrow \Omega_A^{n+1}(X)[n+1].\]
		As $\Omega_A^{n+1}(X)[n]\in\fd(A)^{\leq -n}$, the long exact sequence in cohomology yields that $H^{k}(f^{n})$ is an isomorphism for $k>-n$. Consequently, $\trunc^{>-n}(f^n):\trunc^{>-n}P^{[-n,0]}\stackrel{\sim}\to\trunc^{>-n}X$ is an isomorphism.
    \end{proof}
	
    The following lemma justifies the terminology $(d+1)$-term perfect dg $A$-module.
	
    \begin{lemma}\label{ProjDimd}
		Let $n\in\mathbb{N}$, $P\in\per(A)^{[-n,0]}$ and $Y\in\fd(A)^{\leq 0}$. Then $\Hom_A(P,Y[k])=0$ for $k>n$.
    \end{lemma}
	
    \begin{proof}
		For $n=0$ there is nothing to show. For $n>0$, by definition of $\per(A)^{[-n,0]}$ there is a triangle ${P'\to P\to P^{-n}[n]\to P'[1]}$ with $P'\in\per(A)^{[-(n-1),0]}$, and $P^{-n}\in\add(A)$. Applying $\Hom_A(-,Y[k])$ to this triangle and induction yields the claim.
    \end{proof}
	
    \begin{lemma}\label{TruncIso}
		Let $P\in\per(A)^{(-d,0]}$ and $Y\in\fd(A)$. Then $\trunc^{>-d}$ induces an isomorphism 
		\[\trunc^{>-d}\colon\Hom_A(P,Y)\stackrel{\cong}\longrightarrow \Hom_A(\trunc^{>-d}P,\trunc^{>-d}Y).\]
    \end{lemma}
	
    \begin{proof}
		First apply $\Hom_A(P,-)$ to the canonical triangle $\trunc^{\leq -d}Y\to Y\to \trunc^{>-d}Y\to \trunc^{\leq -d}Y[1]$ to obtain
		\[0=\Hom_A(P,\trunc^{\leq -d}Y)\to \Hom_A(P,Y)\to \Hom_A(P,\trunc^{>-d}Y)\to \Hom_A(P,\trunc^{\leq -d}Y[1])=0,\]
		where the first and the last equality follow by \cref{ProjDimd}. The exactness of the above sequence yields ${\Hom_A(P,Y)\stackrel{\cong}\to \Hom_A(P,\trunc^{>-d}Y)}$. 
		
		Finally, $\trunc^{>-d}$ is defined by the following commutative diagram, where the second isomorphism comes from the adjunction $\trunc^{>-d}\dashv \iota^{>-d}$. 
		\[\begin{tikzcd}
			{\Hom_A(P,Y)} & {\Hom_A(P,\trunc^{>-d}Y)} \\
			{\Hom_A(P,Y)} & {\Hom_A(\trunc^{>-d}P,\trunc^{>-d}Y)}
			\arrow["\cong", from=1-1, to=1-2]
			\arrow[equals,from=1-1, to=2-1]
			\arrow["\cong", from=1-2, to=2-2]
			\arrow["{\trunc^{>-d}}", from=2-1, to=2-2]
		\end{tikzcd}\]
    \end{proof}
	
    We are now ready to prove \cref{d-term-Equiv}.
    \begin{proof}[Proof of \cref{d-term-Equiv}]
		That the functor $\trunc^{>-d}$ is dense, is precisely \cref{d-res-lemma}.
		
		The functor $\trunc^{>-d}$ is full: Let $P,P'\in\per(A)^{[-d,0]}$. Applying $\Hom_A(P,-)$ to the canonical triangle ${\trunc^{\leq -d}P'\to P'\to \trunc^{>-d}P'\to\trunc^{\leq -d}(P')[1]}$ and \cref{ProjDimd} yield the following exact sequence 
		\[\Hom_A(P,P')\to\Hom_A(P,\trunc^{>-d}P')\to\Hom_A(P,\trunc^{\leq -d}(P')[1])=0.\]
		Thus the map 
		\[\Hom_A(P,P')\to \Hom_A(P,\trunc^{>-d}P')\cong \Hom_A(\trunc^{>-d}P,\trunc^{>-d}P')\] 
		is surjective.
		
		Finally, we need to determine the kernel of the map. Clearly, all maps factoring through $\add(A[d])$ are in the kernel. We need to show that the converse holds as well. Hence, let $f\colon P\to P'$ be a morphism in $\per(A)^{[-d,0]}$ such that $\trunc^{>-d}f=0$. By definition of $\per(A)^{[-d,0]}$ there exists a triangle 
		\[P^{d-1}\overset{\iota}{\longrightarrow} P\overset{\pi}{\longrightarrow} P^{-d}[d]\longrightarrow P^{d-1}[1]\] 
		with $P^{-d}\in\add(A)$ and $P^{d-1}\in\per(A)^{(-d,0]}$. Applying the functor $\Hom_A(-,P')$ and \cref{TruncIso} we get the following commutative diagram, with exact top row:
		\[\begin{tikzcd}
			{\Hom_A(P^{-d}[d],P')} & {\Hom_A(P,P')} & {\Hom_A(P^{d-1},P')} \\
			& {\Hom_A(\trunc^{>-d}P,\trunc^{>-d}P')} & {\Hom_A(\trunc^{>-d}P^{d-1},\trunc^{>-d}P')}
			\arrow["{\pi^*}",from=1-1, to=1-2]
			\arrow["{\iota^*}",from=1-2, to=1-3]
			\arrow["\trunc^{>-d}"',from=1-2, to=2-2]
			\arrow["\trunc^{>-d}"',"\cong", from=1-3, to=2-3]
			\arrow["{(\trunc^{>-d}\iota)^*}",from=2-2, to=2-3]
		\end{tikzcd}\]
		We know that $f\in\Hom_A(P,P')$ is sent to $0$ by $\trunc^{>-d}$. By commutativity, it follows that ${f\in\Ker(\iota^*)=\Ima(\pi^*)}$, \emph{i.e.} $f$ factors through $P^{-d}[d]$.
    \end{proof}
	
    As the functor $\trunc^{>-d}$ induces an equivalence $\add(A)\stackrel{\sim}{\to}\add(A)=\Proj(\Hc{d})$, by the isomorphism theorems for additive categories we get the following corollary.
	
    \begin{coro}\label{d-equiv-stable}
		The functor $\trunc^{>-d}$ induces an equivalence
		\[\underline{\trunc^{>-d}}\colon \frac{\per(A)^{[-d,0]}}{\add\left(A[d]\oplus A\right)}\stackrel{\sim}{\longrightarrow} \underline{\Hc{d}}.\]
    \end{coro}
	
    We get the following dual statement of \cref{d-term-Equiv} and \cref{d-equiv-stable}. 
	
    \begin{prop}\label{d-equiv-costable}
		The functor 
		\[[d-1]\circ \trunc^{<d}\colon\thick(\DD A)^{[0,d]}\longrightarrow \Hc{d}\] 
		induces equivalences
		\[[d-1]\circ\trunc^{<d}\colon \frac{\thick(\DD A)^{[0,d]}}{\add\left(\DD A[-d]\right)}\stackrel{\sim}{\longrightarrow} \Hc{d} ,\quad \quad
		\overline{[d-1]\circ\trunc^{<d}}\colon \frac{\thick(\DD A)^{[0,d]}}{\add\left(\DD A[-d]\oplus \DD A\right)}\stackrel{\sim}{\longrightarrow} \overline{\Hc{d}}\]
    \end{prop}
	
    \begin{proof}
		As the proof is dual to the one of \cref{d-term-Equiv} and \cref{d-equiv-stable} we only give a sketch of the dualising procedure. The functor $\DD:\fd(A)^{op}\stackrel{\sim}{\to}\fd(A^{op})$ is $t$-exact with respect to the opposite $t$-structure of the standard $t$-structure on $\fd(A)$, denoted by $t_{op}$, and the standard $t$-structure on $\fd(A^{op})$. Moreover, $\DD A$ is sent to $A^{op}$ by $\DD$. By $t$-exactness we obtain the following commutative diagram, where the equivalence on the right follows from \cref{d-term-Equiv} applied to $A^{op}$:
		\[\begin{tikzcd}
			& {\fd(A)^{op}} & {\fd(A^{op})} \\
			{\frac{\left(\thick(\DD A)^{[0,d]}\right)^{op}}{\add(\DD A[-d])^{op}}} & {\frac{\thick(\DD A)^{[-d,0]_{op}}}{\add(\DD A[d]_{op})}} & {\frac{\per(A^{op})^{[-d,0]}}{\add(A^{op}[d])}} \\
			& {\Hc{d}[-d+1)]^{op}} & {\mathcal{H}^{d}_{A^{op}}}
			\arrow["\DD"', from=1-2, to=1-3]
			\arrow["\sim", draw=none, from=1-2, to=1-3]
			\arrow["{(\trunc^{<d})^{op}}"', from=2-1, to=3-2]
			\arrow[phantom, sloped, "\subseteq", from=2-2, to=1-2]
			\arrow[equals, from=2-2, to=2-1]
			\arrow["\sim", from=2-2, to=2-3]
			\arrow["\DD"', draw=none, from=2-2, to=2-3]
			\arrow["{\trunc^{>-d}_{op}}", from=2-2, to=3-2]
			\arrow[phantom, sloped, "\subseteq", from=2-3, to=1-3]
			\arrow["{\trunc^{>-d}_{A^{op}}}",,from=2-3, to=3-3]
			\arrow["{\sim}"', from=2-3, to=3-3]
			\arrow["{\sim}", from=3-2, to=3-3]
			\arrow["{\DD}"', draw=none, from=3-2, to=3-3]
		\end{tikzcd}\]
        In the diagram above, $[d]_{op}$ and $\thick(\DD A)^{[-d,0]_{op}}$ denote the shift functor and the interval in the opposite triangulated category ${\fd(A)^{op}}$ with the opposite $t$-structure. It follows that $\trunc^{<d}$ induces equivalences
		\[\trunc^{<d}\colon \frac{\thick(\DD A)^{[0,d]}}{\add\left(\DD A[-d]\right)}\stackrel{\sim}{\longrightarrow} \Hc{d} [-d+1] ,\quad \quad 
		[d-1]\circ\trunc^{<d}\colon \frac{\thick(\DD A)^{[0,d]}}{\add\left(\DD A[-d]\right)}\stackrel{\sim}{\longrightarrow} \Hc{d} .\]
		As $[d-1]\circ t^{<d}$ induces an equivalence $\add(\DD A)\stackrel{\sim}{\longrightarrow}\add(\DD A[d-1])=\inj(\Hc{d})$ we get the second claimed equivalence.
    \end{proof}

    Recall from (\ref{TauDef}) that $\define{\tau}\colon\underline{\Hc{d}}\to \overline{\Hc{d}}$ is defined by the commutative diagram:
    \[\begin{tikzcd}
		{\frac{\per(A)^{[-d,0]}}{\add\left(A[d]\oplus A\right)}} & {\frac{\thick(\DD A)^{[0,d]}}{\add\left(\DD A[-d]\oplus \DD A\right)}} \\
		{\underline{\Hc{d}}} & {\overline{\Hc{d}}}
		\arrow["{[-d]\circ \nu_A}", from=1-1, to=1-2]
		\arrow["\sim"', draw=none, from=1-1, to=1-2]
		\arrow["{(\underline{\trunc^{>-d}})^{-}}"', from=2-1, to=1-1]
		\arrow["\sim", draw=none, from=2-1, to=1-1]
		\arrow["{\overline{[d-1]\circ\trunc^{<d}}}", from=1-2, to=2-2]
		\arrow["\sim"', draw=none, from=1-2, to=2-2]
		\arrow["\tau", from=2-1, to=2-2]
    \end{tikzcd}\]

    \begin{proof}[Proof of \cref{MainTheorem}]
		We need to prove that there exists a pair $(\tau,\eta)$ of a $k$-linear equivalence $\tau\colon\underline{\Hc{d}}\to \overline{\Hc{d}}$ and a binatural isomorphism 
		\[\eta_{X,Y}\colon\underline{\Hc{d}}(X,Y)\cong \DD\Hom_A(Y,\tau X[1]) \quad \text{for any }X,Y\in{\Hc{d}}.\]
		As all involved functors in the definition of $\tau$ are $k$-linear equivalences, $\tau$ is a $k$-linear equivalence. It remains to define $\eta$.
		
		Let $X,Y\in \Hc{d}$, $P_X\coloneqq(\underline{\trunc^{>-d}})^{-}X$, and $I_X\coloneqq\nu(P_X)$. By definition of $\tau$ we get the isomorphism $\tau(X)=\trunc^{<d}(I_X[-d])[d-1]\cong\trunc^{<0}(I_X)[-1]$ and thus we obtain the canonical triangle in $\fd(A)$:
		\begin{equation}\label{TauTriangle}
			\tau(X)\longrightarrow I_X[-1]\longrightarrow \trunc^{\geq 0}(I_X)[-1]\longrightarrow \tau(X)[1].
		\end{equation}
		Fix a derived projective cover $p_Y\colon P_Y^0\to Y$ of $Y$. By \cref{StableHom}, there is an exact sequence
		\[0 \leftarrow {\underline{\Hc{d}}(X,Y)} \leftarrow {\Hom_A(X,Y)} \leftarrow {\Hom_A(X,P_Y^0)}.\]
		Moreover, applying $\DD\Hom_A(Y,-)$ to (\ref{TauTriangle}) we obtain the following exact sequence:
		\[0={\DD\Hom_A(Y,\trunc^{\geq 0}(I_X)[-1])} \leftarrow {\DD\Hom_A(Y,\tau(X)[1])} \leftarrow {\DD\Hom_A(Y,I_X)} \leftarrow {\DD\Hom_A(Y,\trunc^{\geq 0}(I_X))}.\]
		Finally, we have the following commutative diagram 
		\[\begin{tikzcd}
			&& {0=\DD\Hom_A(\cone(p_Y),\trunc^{\geq 0}(I_X))} \\
			& {\DD\Hom_A(Y,I_X)} & {\DD\Hom_A(Y,\trunc^{\geq 0}(I_X))} \\
			{\DD\Hom_A(P_Y^0, \tau(X)[1])=0} & {\DD\Hom_A(P_Y^0, I_X)} & {\DD\Hom_A(P_Y^0,\trunc^{\geq 0}(I_X))}
			\arrow[from=2-3, to=1-3]
			\arrow["{\DD\Hom_A(Y,\iota)}"', from=2-3, to=2-2]
			\arrow["{\DD\Hom_A(p,I_X)}", from=3-2, to=2-2]
			\arrow[from=3-2, to=3-1]
			\arrow[two heads, from=3-3, to=2-3]
			\arrow[two heads, from=3-3, to=3-2]
		\end{tikzcd}\]
		which implies that $\Ima(\DD\Hom_A(p,I_X))=\Ima(\DD\Hom_A(Y,\iota))$. Hence, all together we get the commutative diagram below. The induced morphism on the cokernel is an isomorphism by the universal property of the cokernel.
		\[\begin{tikzcd}
			0 & {\underline{\Hc{d}}(X,Y)} & {\Hom_A(X,Y)} & {\Hom_A(X,P_Y^0)} \\
			&& {\Hom_A(P_X,Y)} & {\Hom_A(P_X,P_Y^0)} \\
			0 & {\DD\Hom_A(Y,\tau(X)[1])} & {\DD\Hom_A(Y,I_X)} & {\DD\Hom_A(P_Y^0, I_X)}
			\arrow[from=1-2, to=1-1]
			\arrow["\exists ! \; \eta_{X,Y}"',"\cong", dashed, from=1-2, to=3-2]
			\arrow[from=1-3, to=1-2]
			\arrow["\text{adj.}"',"\cong", from=1-3, to=2-3]
			\arrow[from=1-4, to=1-3]
			\arrow["\text{adj.}"',"\cong", from=1-4, to=2-4]
			\arrow["(\ref{AR-formula})"',"\cong", from=2-3, to=3-3]
			\arrow[from=2-4, to=2-3]
			\arrow["(\ref{AR-formula})"',"\cong", from=2-4, to=3-4]
			\arrow[from=3-2, to=3-1]
			\arrow[from=3-3, to=3-2]
			\arrow[from=3-4, to=3-3]
		\end{tikzcd}\]
		As $\eta_{X,Y}$ is induced by a natural transformation it is a natural isomorphism.
    \end{proof}

    \section{Auslander--Reiten quivers}\label{AR-quivers}

    In this section, we describe properties of the Auslander--Reiten quiver of $\Hc{d}$. We briefly recall the necessary definitions and statements from \cite{INP24}. We call a morphism $f\colon M\to N$ in $\Hc{d}$
    \begin{itemize}
        \item [(i)] \definef{left minimal} if each morphism $g\colon N\to N$ satisfying $g\circ f=f$ is an isomorphism. \definef{Right minimal} morphisms are defined dually. 
        \item [(ii)] \definef{left almost split} if it is not a section and each morphism $h\colon M\to N'$ which is not a section factors $f$. \definef{Right almost split} morphisms are defined dually. 
        \item [(iii)] $f$ is \definef{source} if it is left almost split and left minimal. \definef{Sink} morphisms are defined dually. 
    \end{itemize}

    \begin{theorem}[{\cite[Theorem 2.9]{INP24}}]\label{thm: iyama nakaoka palu}
        Let $L\xrightarrow{f} M\xrightarrow{g} N\dashrightarrow$ be a conflation in $\Hc{d}$. Then the following conditions are equivalent: 
        \begin{itemize}
            \item [(i)] $f$ is left almost split and $N$ is indecomposable.
            \item [(ii)] $g$ is right almost split and $L$ is indecomposable.
            \item [(iii)] $f$ is a source morphism.
            \item [(iv)] $g$ is a sink morphism.
        \end{itemize}
    \end{theorem}

    \begin{definition}
        We call a conflation $L\xrightarrow{f} M\xrightarrow{g} N\dashrightarrow$ in $\Hc{d}$ an \definef{almost-split conflation} if the equivalent condition in \cref{thm: iyama nakaoka palu} are satisfied. 
    \end{definition}

    By combining \cref{AlmostSplitCoro} with \cite[Prop.~3.12]{INP24}, we obtain the following corollary.
    \begin{coro}\label{LocalAR} 
        The following statements hold:
        \begin{enumerate}
            \item For a non-projective indecomposable object $M\in\Hc{d}$, we have an almost-split conflation
            \[\tau M\to X\to M\dashrightarrow.\]
            \item For a non-injective indecomposable object $N\in\Hc{d}$, we have an almost-split conflation
            \[N\to Y\to \tau^{-1}N\dashrightarrow.\]
        \end{enumerate}
        Moreover, by defining for all indecomposable $L,L'\in\Hc{d}$,
        \begin{align*}
             D_L\coloneqq(\Hc{d}/\rad\Hc{d})(L,L), \quad & \Irr(L,L')\coloneqq (\rad\Hc{d}/\rad^2\Hc{d})(L,L'),\\ 
            d_{LL'}\coloneqq \dim \Irr(L,L')_{D_L}, \quad  & d'_{LL'}\coloneqq _{D_{L'}}\Irr(L,L'),
        \end{align*}
        the objects $X$ and $Y$ in the above can be computed as 
        \[\bigoplus_{L:\text{ indec.}}L^{\oplus d'_{\tau (M)L}}\cong X\cong\bigoplus_{L:\text{ indec.}}L^{\oplus d_{LM}}, \quad 
        \bigoplus_{L:\text{ indec.}}L^{\oplus d'_{NL}}\cong Y\cong\bigoplus_{L:\text{ indec.}}L^{\oplus d_{L\tau^{-1}(N)}}.\]
    \end{coro}

    From the last formula, we immediately deduce the following corollary.
    
    \begin{coro}\label{prop: iiyatsu}
        Let $M$ be a non-projective indecomposable object in $\Hc{d}$ and $N\in\Hc{d}$. Then there is an irreducible morphism $f\in\Irr(N,M)$ if and only if there is an irreducible morphism $\tau M\to N$. 
    \end{coro}

    \begin{definition}
        For $M\in\Hc{d}$, we define the \definef{top} of $M$ by $\Top M \coloneqq \Top H^0(M)\in\Hc{d}$. Dually, the \definef{socle} of $M$ is defined by $\soc M:=\soc(H^{-n+1}M)[d-1]\in\Hc{d}$. Finally, we define the \definef{radical} $\rad M\in\Hc{d}$ of $M$ as $\rad M:=\Cocone(M\to \Top M)$ and $M/\soc M:=\cone(\soc M\to M)$. By construction, there are the following triangles in $\fd(A)$.
        \begin{equation}\label{radical}
            \rad(M)\to M\to\Top M \to \rad(M)[1], \quad \soc(M)\to M\to M/\soc M \to \soc(M)[1]
        \end{equation}
    \end{definition}

    \begin{prop}\label{prop: sink of proj}
        Let $P\in\proj(\Hc{d})$ be indecomposable. Then $\rad P\to P$ is a sink morphism. 
    \end{prop}
    
    \begin{proof}
        \leaveout{Since $\rad P\to P$ is a $1$-monomorphism in $\Hc{d}$, it is a monomorphism and thus is a right minimal. Next, we show that $\rad P\to P$ is unique maximal $1$-monomorphism up to isomorphisms. Take a maximal $1$-monomorphism $M\to P$ in $\modn\Lambda$ which is not an isomorphisms. Then the induced morphism $\top M\to \top P$ equals to zero. Thus we have a morphism $M\to P$ which commutes a following diagram: 
        $$
        \begin{tikzcd}
        {} & M & \top (H^0M)\\
        \rad P & P & \top (H^0P)
        \Ar{1-2}{1-3}{}
        \Ar{2-1}{2-2}{}
        \Ar{2-2}{2-3}{}
        \Ar{1-3}{2-3}{"0"}
        \Ar{1-2}{2-2}{}
        \Ar{1-2}{2-1}{dashed}
        \end{tikzcd}
        $$
        And the induced morphim $M\to \rad P$ is isomorphisms since $H^0M\to H^0P$ isomorphic to $\rad (H^0P)\to H^0P$. Thus, $\rad P\to P$ is unique maximal $1$-monomorphism. 
        
        At the last, we show that $\rad P\to P$ is right almost split. Take a $N\to P$ which is not a retraction (equivalently,not a $n$-epimorphism). Then there is a unique factorization $N\xrightarrow{e_n}N'\xrightarrow{m_1}P$ by Theorem \ref{prop: fundamental propositions of n-extended} (iii'). By the uniqueness of the factorization, $m_1\colon N'\to P$ is not a isomorphism. Thus, $m_1$ factors $\rad P\to P$.}

        First, $\iota_P:\rad P\to P$ is not a retract as it not a deflation. Next, we check it is right minimal. Let $h:\rad P\to \rad P$ be such that $\iota_P \circ h=\iota_P$. Then $h$ induces an isomorphism in cohomology since $H^i(\iota_P)$ is an isomorphism for $i\neq 0$ and a monomorphism for $i=0$. But this means $h$ is an isomorphism in $\fd(A)$. Finally, let $f:N\to P$ be a morphism in $\Hc{d}$ that is not a retract, and denote by $p:P\to \Top P$ the morphism from the triangle (\ref{radical}). Then, since $P$ is projective, $f$ is not a deflation and therefore $p \circ f=0$ as $\Top P=\Top H^0(P)$ is indecomposable. Applying $\Hom_A(N,-)$ to (\ref{radical}) we obtain the exact sequence
        \[\Hom_A(N,\rad P)\stackrel{(\iota_P)_*}{\longrightarrow}\Hom_A(N,P)\stackrel{p_*}{\longrightarrow}\Hom_A(N,\Top P).\]
        As $p_*(f)=0$, by exactness $f$ factors through $\iota_P$
    \end{proof}

    Dually, we have the following proposition. 
    
    \begin{prop}\label{prop: sorce inj}
        Let $I\in\inj(\Hc{d})$ be indecomposable. Then $I\to I/\soc I$ is a source morphism. 
    \end{prop}

    \begin{prop}\label{proj inj}
        Let $P\in\proj(\Hc{d})\cap\inj(\Hc{d})$ be indecomposable, non-simple. We have an almost-split conflation: 
        \[\rad P\stackrel{\begin{psmallmatrix}\iota\\ q\end{psmallmatrix}}{\longrightarrow} P\oplus (\rad P/\soc P)\stackrel{\begin{psmallmatrix}p &-j\end{psmallmatrix}}{\longrightarrow} P/\soc P\dashrightarrow,\]
        where $\iota,j,p,q$ are the canonical inclusions and projections from (\ref{radical}) respectively.
    \end{prop}
    
    \begin{proof}
        One easily checks that the above defines a conflation in $\Hc{d}$. Observe that $\rad(P)$ is indecomposable because it has simple socle and dually $P/\soc P$ is indecomposable as $P$ is indecomposable.\leaveout{ By \cref{prop: sink of proj}, $\iota$ is the unique irreducible morphism ending in $P$. Dually, by \cref{prop: sorce inj}, $p$ is the unique irreducible morphism starting in $P$. Hence \cref{prop: iiyatsu} yields $\tau (P/\soc P)\cong \rad P$.} We need to argue that $\begin{psmallmatrix}\iota\\q\end{psmallmatrix}$ is left almost split. Clearly it is not a section, as the given sequence is not split. Consider a morphism $h:\rad P\to M$, which is not a section. First we consider the case that $h$ is an inflation. We claim that $h$ factors through $\iota$. By injectivity of $P$, there exists $g:M\to P$ such that $\iota=g\circ h$. If the composition $M\stackrel{g}{\to}P\to\Top P$ is $0$, the morphism $g$ factors through $\iota$, that is there exists $\tilde{g}:M\to\rad P$ such that $g=\tilde{g}\circ\iota$. 
        \[\begin{tikzcd}
	M & M & 0 & {} \\
	{\rad P} & P & {\Top P} & {}
	\arrow["\id", from=1-1, to=1-2]
	\arrow["{\exists \tilde{g}}", shift left=2, dotted, from=1-1, to=2-1]
	\arrow[from=1-2, to=1-3]
	\arrow["g", from=1-2, to=2-2]
	\arrow[dashed, from=1-3, to=1-4]
	\arrow[from=1-3, to=2-3]
	\arrow["h", shift left=2, from=2-1, to=1-1]
	\arrow["\iota", from=2-1, to=2-2]
	\arrow[from=2-2, to=2-3]
	\arrow[dashed, from=2-3, to=2-4]
        \end{tikzcd}\]
        We claim that $\tilde{g}$ is a retraction of $h$. First observe that $\iota\circ \tilde{g}\circ h=h\circ \id \circ g=\iota$. Furthermore, as $\Top (P)[-1]\in\fd(A)^{>0}$, applying $\Hom_A(\rad P,-)$ to the second row yields the exact sequence
        \[0=\Hom_A(\rad P,\Top (P) [-1])\longrightarrow \Hom_A(\rad P,\rad P)\stackrel{\iota_*}{\longrightarrow}\Hom_A(\rad P, P).\]
        Hence, $\iota_*$ is injective and $\tilde{g}\circ h=\id$, that is $h$ is a section, a contradiction to the assumption.
        
        If the composition $M\stackrel{g}{\to}P\to\Top P$ is non-zero, then $g$ is a deflation, as $\Top P$ is indecomposable. By projectivity of $P$ we get a section $s$ of $g$ and hence $s\circ \iota =s\circ g\circ h=h$, that is $h$ factors through $\iota$, which is what we wanted to prove.
        
        Assume now $h$ is not an inflation. The the equality $h(\soc P)=0$ holds, since $\soc P$ is simple. Hence, $h$ factors through $\rad P/\soc P$.
    \end{proof}

    Next, we discuss the connected components of the Auslander–Reiten quiver. The goal is to prove the 1st Brauer--Thrall conjecture for $\Hc{d}$. The main part consists in proving the Harada--Sai Lemma in this context. The general strategy is an adaption of \cite[Ch.~IV,Lem.~5.2]{ASS06}. First we need the following variant of \cite[Thm.~4.21]{M25}.

    \begin{lemma}\label{lem: mochizuki}
        Let $-d<\alpha\leq 0$. A morphism $f\colon M\to N$ in $\Hc{d}$, admits a factorization
        \[\begin{tikzcd}
	M && N \\
	& {\Ima^{\alpha}(f)}
	\arrow["f", from=1-1, to=1-3]
	\arrow["\pi"', from=1-1, to=2-2]
	\arrow["\iota"', from=2-2, to=1-3]
        \end{tikzcd}\]
        such that
        \begin{itemize}
            \item $H^{i}(\pi)$ is an isomorphism for $\alpha <i \leq 0$. 
            \item $H^{\alpha}(\pi)$ is an epimorphism, $H^{\alpha}(\iota)$ is a monomorphism and $H^{\alpha}(\iota)$ induces an isomorphism $H^{\alpha}(\Ima^{\alpha}(f))\cong\Ima(H^{\alpha}(f))$.
            \item $H^{i}(\iota)$ is an isomorphism for $-d<i<\alpha$.
        \end{itemize}
    \end{lemma}

    \begin{proof}
        Consider the truncation 
        \[\cone f \to \trunc^{\geq \alpha} (\cone f),\]
        and define $\Ima^{\alpha}(f)$ as the cocone of the composition $N\to\cone f\to \trunc^{\geq \alpha}(\cone f)$. By construction, there is a triangle
        \[\Ima^{\alpha}(f) \to N \to t^{\geq \alpha} (\cone f) \dashrightarrow.\]
        This gives rise to a morphism \( M \to \Ima^{\alpha}(f) \). Now consider the following diagram, in which the two top rows and the rightmost two columns are part of triangles.
        \[
        \begin{tikzcd}
        M \arrow[equal]{d} \arrow["\pi"]{r} & \Ima^{\alpha}(f) \arrow{r} \arrow["\iota"]{d} & t^{\leq \alpha-1} (\cone f) \arrow{d} \arrow[dashed]{r} & {}\\
        M \arrow["f"]{r} \arrow["0"']{rd} & N \arrow{r} \arrow{d} & \cone f \arrow{d} \arrow[dashed]{r} & {}\\
        & t^{\geq \alpha} (\cone f) \arrow[equal]{r} \arrow[dashed]{d} & t^{\geq \alpha} (\cone f) \arrow[dashed]{d}\\
        {} & {} & {} & {}
        \end{tikzcd}
        \]
        Then $\Ima^{\alpha}\in\Hc{d}$, since $M,\trunc^{\leq \alpha-1}(\cone f)\in\Hc{d}$. The long exact sequence in cohomology for the first row yields that $H^i(\pi)$ is an isomorphism for $\alpha<i\leq 0$ and an epimorphism for $i=\alpha$. The long exact sequence in cohomology for the middle column yields that $H^{i}(\iota)$ is an isomorphism for $-d<i<\alpha$ and a monomorphism for $i=\alpha$. As $f=\iota\pi$ it is also clear that $H^{\alpha}(\iota)$ induces an isomorphism $H^{\alpha}(\Ima^{\alpha}(f))\cong\Ima(H^{\alpha}(f))$.
    \end{proof}

    \begin{definition}
        Fix $b\in\mathbb{N}$. A \definef{$b$-Harada--Sai factorization} is a sequence
        \[M_1\stackrel{f_1}{\longrightarrow} M_2\stackrel{f_2}{\longrightarrow}\cdots \stackrel{f_{t-1}}{\longrightarrow}M_{t}\]
        of non-isomorphisms $f_i:M_i\to M_{i+1}$ between indecomposable objects $M_1,\ldots,M_t\in\Hc{d}$ with 
        \[\dim_kH^{*}(M_i)\leq b.\] 
        We will sometimes refer to a $b$-Harada--Sai factorization by $f=f_{t-1}\circ \ldots \circ f_1$. We call the the natural number $t$ the \definef{length} of $f$ and define the \definef{rank} of $f$ by 
        \[\rk(f)\coloneqq \min\left\{\dim_kH^*(F):F\in\Hc{d},f=hg \; \text{ for some } g\in\Hc{d}(M_1,F),h\in\Hc{d}(F,M_t)\right\}.\]
    \end{definition}

    \begin{lemma}\label{FixedDimension}
        Let $b\in\mathbb{N}$ and $t(1)=b(d-1)+2$. Consider a sequence of non-isomorphisms
        \[M_1\stackrel{f_1}{\longrightarrow} M_2\stackrel{f_2}{\longrightarrow}\cdots \stackrel{f_{t(1)-1}}{\longrightarrow}M_{t(1)}\]
        between (not necessarily indecomposable) objects $M_i\in\Hc{d}$ such that ${\dim_kH^*(M_i)=b}$. Then it holds $\rk(f_{t(1)-1}\circ\cdots\circ f_1)\leq b-1$.
    \end{lemma}

    \begin{proof}
        If $f_1$ is a degreewise monomorphism in cohomology, for dimension reasons, it is an isomorphism, a contradiction to the assumption. Therefore, we choose the maximal $-d<\alpha_1\leq 0$ such that $H^{\alpha_1}(f_1)$ is not a monomorphism. Consider the factorization from \cref{lem: mochizuki}
        \[f_1\colon M_1\stackrel{\pi_1}{\longrightarrow}\Ima^{\alpha_1}(f_1)\stackrel{\iota_1}{\longrightarrow} M_2.\]
        By \cref{lem: mochizuki} it holds 
        \[\dim_kH^j(\Ima^{\alpha_1}(f_1))=\begin{cases}\dim_kH^j(M_1) & \text{ if }\alpha_1< j\leq 0,\\ \dim_k\Ima (H^j(f_1)) & \text{ if }j=\alpha_1,\\ \dim_kH^j(M_2) & \text{ if }-d<j<\alpha_1.\end{cases}\]
        In particular, if we associate to an object $M\in\Hc{d}$ the vector 
        \[\dim(M)\coloneqq(\dim_kH^{j}(M))_{-d<j\leq 0}\in\mathbb{N}^d\]
        we see that in the interval $[\alpha_1,0]$ the cohomology strictly decreases in dimension, that is
        \[\dim(\Ima^{\alpha_1}(f_1))_{\alpha_1< j\leq 0}\leq \dim(M_1)_{\alpha_1<j\leq 0}, \quad \dim(\Ima^{\alpha_1}(f_1))_{\alpha_1}< \dim(M_1)_{\alpha_1}.\]
        Moreover, one easily sees that $\dim_kH^*(\Ima^{\alpha_1}(f_1))\leq b$. If the inequality is strict, we are done. Otherwise, consider the morphism $f_2\iota_1$. This is not an isomorphism, otherwise $f_2$ would be a split epimorphism and for dimension reasons $f_2$ would be an isomorphism, a contradiction. Hence, we can find a maximal $-d<\alpha_2\leq 0$ such that $H^{\alpha_2}(f_2\iota_1)$ is not a monomorphism. \cref{lem: mochizuki} yields a factorization
        \[\begin{tikzcd}
	{M_1} && {M_2} && {M_3} \\
	& {\Ima^{\alpha_1}(f_1)} && {\Ima^{\alpha_2}(f_2\iota_1)}
	\arrow["{f_1}", from=1-1, to=1-3]
	\arrow["{\pi_1}", from=1-1, to=2-2]
	\arrow["{f_2}", from=1-3, to=1-5]
	\arrow["{\iota_1}", from=2-2, to=1-3]
	\arrow["{\pi_2}", from=2-2, to=2-4]
	\arrow["{\iota_2}", from=2-4, to=1-5]
        \end{tikzcd}\]
        As above we see that
        \[\dim(\Ima^{\alpha_2}(f_2\iota_1))_{\alpha_2<j\leq 0}\leq \dim(\Ima^{\alpha_1}(f_1))_{\alpha_2 < j\leq 0}, \quad \dim(\Ima^{\alpha_2}(f_2\iota_1))_{\alpha_2}< \dim(\Ima^{\alpha_1}(f_1))_{\alpha_2}.\]
        Continuing in this manner either we obtain $\dim_{k}H^*(\Ima(f_i\iota_{i-1}))<b$ at some point or after at most $b(d-1)$ iterations we have $\dim_kH^{-d+1}(\Ima(f_i\iota_{i-1}))=b$. In that case $\alpha_{i+1}=1-d$, and since $f_{i+1}\iota_i$ is not an isomorphism $\dim_kH^*(\Ima^{-d+1}(f_{i+1}\iota_i))<b$.
    \end{proof}
    
    \begin{lemma}
        Let $b,n\in\mathbb{N}$ such that $n\leq b$. For a $b$-Harada--Sai factorization 
        \[M_1\stackrel{f_1}{\longrightarrow} M_2\stackrel{f_2}{\longrightarrow}\cdots \stackrel{f_{t(n)-1}}{\longrightarrow}M_{t(n)}\]
        of length $t(n)=(b(d-1)+2)^n$, it holds $\rk(f)\leq b-n$.
    \end{lemma}

    \begin{proof}
        We prove the statement by induction on $n$. For $n=1$ it holds $t(1)=b(d-1)+2$ and we want to show $\rk(f)\leq b-1$. If there exists an $i$ such that $\dim_kH^*(M_i)<b$ we can simply take $F\coloneqq M_i$. Thus we might assume $\dim_kH^*(M_i)=b$ for all $1\leq i\leq t$. Now \cref{FixedDimension} yields the claim.

        Assuming we already proved the statement for $n-1$, we want to prove it for $n$. We first observe that by induction hypothesis there exist factorizations
        \[\begin{tikzcd}
	{M_1} & {M_{t(n-1)}} & {M_{2t(n-1)-1}} & \cdots & {M_{t(1)t(n-1)-(t(1)-1)}} \\
	& {F_1} & {F_2} & \cdots & {F_{t(1)}}
	\arrow["{\tilde{f_{1}}}", from=1-1, to=1-2]
	\arrow["{g_1}"', from=1-1, to=2-2]
	\arrow["{\tilde{f}_{2}}", from=1-2, to=1-3]
	\arrow["{g_2}"'{pos=0.4}, from=1-2, to=2-3]
	\arrow["{\tilde{f_3}}", from=1-3, to=1-4]
	\arrow["{g_3}"'{pos=0.4}, from=1-3, to=2-4]
	\arrow["{\tilde{f}_{t(1)-1}}", from=1-4, to=1-5]
	\arrow["{g_{t(1)}}"'{pos=0.4}, from=1-4, to=2-5]
	\arrow["{h_1}", from=2-2, to=1-2]
	\arrow["{g_2h_1}"', from=2-2, to=2-3]
	\arrow["{h_2}", from=2-3, to=1-3]
	\arrow["{g_3h_2}"', from=2-3, to=2-4]
	\arrow["{g_{t(1)}h_{t(1)-1}}"', from=2-4, to=2-5]
	\arrow["{h_{t(1)}}", from=2-5, to=1-5]
        \end{tikzcd}\]
         with $\dim_kH^*(F_i)\leq b-(n-1)$ and $\tilde{f_i}$ is the composition of the $f_j$ in the respective range. If there exists $i$ such that $\dim_kH^*(F_i)\leq b-n$ we can choose $F\coloneqq F_i$. Thus we may assume that $\dim_kH^*(F_i)=b-(n-1)$. If none of the $g_{i+1}h_i$ are isomorphisms we can apply \cref{FixedDimension} for $b'\coloneqq b-(n-1)$ to this sequence and obtain $\rk(f)\leq b-n$. Otherwise, there exists $i$ such that $g_{i+1}h_i$ is an isomorphism. Hence $h_i$ is a split monomorphism and since $M_{i\cdot t(n-1)-(i-1)}$ is indecomposable, $h_i$ is an isomorphism and in particular $\dim_k H^*( M_{i\cdot t(n-1)-(i-1)})=b-(n-1)$.

         Observe that the above argument works analogously if instead of $M_{1},M_{t(n-1)},\ldots$ we consider $M_k,M_{k+t(n-1)},\ldots $ for $1\leq k\leq (b-(n-1))(d-1)< t(n-1)$. Therefore, we get for each such $k$ an $i_k\neq i_{k'}$ for $k'\neq k$ such that $\dim_kH^*(M_{i_k})=b-(n-1)$. Consider the subsequence of the original sequence which only contains these $M_{i_k}$, obtained by composing the respective $f_j$ between them. This is a sequence of non-isomorphisms as otherwise we would get a split epimorphism and by indecomposability an isomorphism. Thus we can again apply \cref{FixedDimension} to this sequence for $b'=b-(n-1)$ and obtain $\rk(f)\leq b-n$.
    \end{proof}

    \begin{rema}
        Considering the proof, one could reduce the bound by definining for $b,n\in\mathbb{N}$ inductively 
        \[t(b,1)=(d-1)b+2, \quad \text{and} \quad t(b,n+1)=t(b-n,1)\cdot t(b,n).\]
        For readability, and since we are mainly interested in a finite bound, we decided to not state this sharper bound. For the case $d=1$, this bound is known to be sharp, for $d>1$ we do not know if this is the case as well.
    \end{rema}

    Applying the above for $b=n$ we immediately get the following corollary.
    
    \begin{coro}[Harada--Sai Lemma]\label{Harada-Sai}
        Let $b\in\mathbb{N}$ be natural number and
        \[M_1\stackrel{f_1}{\longrightarrow} M_2\stackrel{f_2}{\longrightarrow}\cdots \stackrel{f_{t-1}}{\longrightarrow}M_{t}\]
        a sequence of non-isomorphisms in $\Hc{d}$ such that $M_i$ is indecomposable and $\dim_k H^*(M_i)\leq b$. If the inequality $t \geq (b(d-1)+2)^b\eqqcolon l(b)$ holds, then $f_{t-1}\circ \cdots \circ f_{1}= 0$
    \end{coro}

    In the following, we call a sequence of irreducible morphisms in $\Hc{d}$ of the form
    \[M_1\stackrel{f_1}{\longrightarrow} M_2\stackrel{f_2}{\longrightarrow}\cdots \stackrel{f_{t-1}}{\longrightarrow}M_{t}\]
    with $M_i$ indecomposable, a \definef{sequence of irreducible morphisms} from $M_1$ to $M_t$ of length $t$.

    \begin{lemma}[{\cite[Ch.~IV,Lem.~5.1]{ASS06}}]\label{lem: nazo}
        Let $M,N\in\Hc{d}$ be indecomposable with $\Hc{d}(M,N)\neq 0$. Assume there is no sequence of irreducible morphisms from $M$ to $N$ of length $<t$. 
        \begin{enumerate}
            \item There exists a sequence of irreducible morphisms in $\Hc{d}$
            \[M=M_1\stackrel{f_1}{\longrightarrow} M_2\stackrel{f_2}{\longrightarrow}\cdots \stackrel{f_{t-1}}{\longrightarrow}M_t\]
            and a morphism $g\colon M_t\to N$ in $\Hc{d}$ such that $g\circ f_{t-1}\circ \cdots f_1\neq 0$.
            \item There exists a sequence of irreducible morphisms in $\Hc{d}$
            \[N_t\stackrel{g_{t-1}}{\longrightarrow} N_{t-1}\stackrel{g_{t-2}}{\longrightarrow}\cdots \stackrel{g_{1}}{\longrightarrow}N_1=N\]
            and a morphism $f\colon M\to N_t$ in $\Hc{d}$ such that $g_1\circ\cdots\circ g_{t-1}\circ f\neq 0$.
        \end{enumerate}
    \end{lemma}

    The above statement holds more generally for any Hom-finite Krull–Schmidt category that has source morphisms. The proof is exactly the same as in \cite{ASS06}.

    \begin{lemma}\label{PathLemma}
        Let $\C$ be a \definef{bounded} connected component of the Auslander--Reiten quiver of $\Hc{d}$, that is there exists $b\in\mathbb{N}$, such that $\dim_kH^*(M)\leq b$ for all $[M]\in\C$. Let $M,N\in\Hc{d}$ be indecomposable such that $\Hc{d}(M,N)\neq 0$ and either $[M]\in\C$ or $[N]\in \C$. Then both $[M],[N]\in\C$ and there exists a path of length $<l(b)$ in $\C$ from $[M]$ to $[N]$.
    \end{lemma}

    \begin{proof}
        We consider the case $[M]\in\C$, the other case is dual. Assume there was no sequence of irreducible morphisms of length $<l(b)$ from $M$ to $N$. By \cref{lem: nazo}, there would be a sequence $f=f_{l(b)-1}\circ\cdots\circ f_1\neq 0$ of irreducible morphisms starting in $M$. However, since $\C$ is connected, every object $M_k$ in this sequence is in $\C$ and therefore $\dim_kH^*(M_k)\leq b$. By \cref{Harada-Sai} it holds $f=0$, a contradiction. Hence both $[M],[N]\in\C$ and they are connected by a path of length $<l(b)$ in $\C$.
    \end{proof}

    We get the following analogue of the 1st Brauer--Thrall Conjecture for extended hearts. The proof is similar to {\cite[Ch.~IV,Thm.~5.4]{ASS06}}.

    \begin{theorem}\label{ConnectedTheorem}
        Let $\C_1,\C_2,\dots,\C_m$ be bounded connected components of the Auslander--Reiten quiver of $\Hc{d}$. Let $\C\coloneqq\bigcup_iC_i$ and assume the set of vertices
        \[\left\{\ [t^{> -d}P[i]]\ \colon  P\in\Proj(\Hc{d}) \text{ is indecomposable and } 0\leq i\leq 1-d\  \right\}\]
        is contained in $\C$. Then $\C$ is finite and $\C$ is the whole Auslander--Reiten quiver of $\Hc{d}$. 
    \end{theorem}  
    
    \begin{proof}
        Let $N\neq 0\in\Hc{d}$. Then there exists $P\in\proj(\Hc{d})$ indecomposable and $0\leq i\leq d-1$ such that $\Hc{d}(t^{>-d}P[i],N)\neq 0$. By \cref{PathLemma}, it holds $[N]\in\C$ since $t^{>-d}P[i]$ is contained in $\C_i$ for some $i$ by assumption. Thus, it follows that $\C$ is the full Auslander--Reiten quiver.  

        It remains to show that $\C_i$ is finite. By assumption there exists $b\in\mathbb{N}$ with $\dim_kH^*(M)\leq b$ for all $[M]\in \C$. For every $[N]\in\C_i$, by \cref{PathLemma} there exists a path of length $<l(b)$ from $[\trunc^{>d}(P[i])]$ to $[N]$. However, there are only finitely many indecomposable projectives and the Auslander--Reiten quiver is a locally finite quiver, hence there are only finitely many paths of length $<l(b)$ starting in an object of the form $\trunc^{>-d}(P[i])$. Thus $\C_i$ and hence $\C$ is finite.
    \end{proof}

    \begin{definition}
        We call $A$ \definef{$d$-connected} if the set
        \[\left\{\ [t^{> -d}P[i]]\ \colon  P\in\Proj(\Hc{d}) \text{ is indecomposable and } 0\leq i\leq 1-d\  \right\}\]
        is connected in $\Hc{d}$, that is between any two objects $X,Y$ in this set there exists a zig-zag of non-zero morphisms between indecomposable objects $X_k\in\Hc{d}$. We will denote this by 
        \[\begin{tikzcd}
        	X & {X_1} & {X_2} & \cdots & {X_n} & Y
        	\arrow[no head, from=1-1, to=1-2]
        	\arrow[no head, from=1-2, to=1-3]
        	\arrow[no head, from=1-3, to=1-4]
        	\arrow[no head, from=1-4, to=1-5]
        	\arrow[no head, from=1-5, to=1-6]
        \end{tikzcd}.\]
        Here the arrows without orientation represent a non-zero morphism with some orientation.
    \end{definition}

    \begin{rema}
        $A$ is $1$-connected if and only if $H^0(A)$ is connected. Indeed, for any zig-zag of morphisms we can replace each object in $\Mod(H^0(A))$ by its projective cover to obtain a zig-zag that only contains indecomposable projectives. 
        \[\begin{tikzcd}
        	& {P_1} & {P_2} & \cdots & {P_n} \\
        	P & {X_1} & {X_2} & \cdots & {X_n} & Q
        	\arrow[no head, from=1-2, to=1-3]
        	\arrow[two heads, from=1-2, to=2-2]
        	\arrow[no head, from=1-3, to=1-4]
        	\arrow[two heads, from=1-3, to=2-3]
        	\arrow[no head, from=1-4, to=1-5]
        	\arrow[two heads, from=1-5, to=2-5]
        	\arrow[no head, from=1-5, to=2-6]
        	\arrow[no head, from=2-1, to=1-2]
        	\arrow[no head, from=2-1, to=2-2]
        	\arrow[no head, from=2-2, to=2-3]
        	\arrow[no head, from=2-3, to=2-4]
        	\arrow[no head, from=2-4, to=2-5]
        	\arrow[no head, from=2-5, to=2-6]
        \end{tikzcd}\]
    \end{rema}

    \begin{theorem}\label{BT1}
        Let $A$ be $d$-connected and $\C$ a bounded connected component of the Auslander--Reiten quiver of $\Hc{d}$. Then $\C$ is finite and $\C$ is the whole of the Auslander--Reiten quiver of $\Hc{d}$.
    \end{theorem}

    \begin{proof}
        We need to show that it holds $\trunc^{>-d}(P[i])\in\C$ for $P\in\Proj(\Hc{d})$ indecomposable and ${0\leq i<d}$. For $[N]\in\C$ there exists $P(N)\in\Proj(\Hc{d})$ indecomposable and $0\leq k<d$ such that ${\Hc{d}(\trunc^{>-d}(P(N)[k]),N)\neq 0}$. By \cref{PathLemma} $[\trunc^{>-d}(P(N)[k])]\in \C$. By $d$-connectedness and \cref{PathLemma} all $\trunc^{>-d}(P[i])$ are in $\C$.
    \end{proof}

    \begin{rema}
        In \cref{ExNot2Conn} we define a dg algebra that is not $2$-connected and not the product of $2$-connected algebras but satisfies the assumptions of \cref{ConnectedTheorem}. Hence \cref{ConnectedTheorem} is indeed more general than \cref{BT1}.
    \end{rema}

    The following proposition gives many examples of $d$-connected algebras.

    \begin{prop}\label{PropExamplesdConn}
        Assume that $H^0(A)$ is connected. If one of the following conditions holds, then $A$ is $d$-connected.
        \begin{enumerate}
            \item\label{ItemH} It holds $H^{-1}(A)\neq 0$.
            \item\label{ItemSimple} For the sum of simple $H^0(A)$-modules $S\coloneqq \bigoplus_{i=1}^nS_i$ it holds $\Hc{d}(S,S[1])\neq 0$.
        \end{enumerate}
    \end{prop}

    \begin{proof}
        We first observe that as $\proj(\Hc{d})\simeq\mathrm{proj}(H^0(A))$, it follows that $\proj(\Hc{d})$ is connected. By \cref{TruncIso} it holds for all $P,P'\in\proj(\Hc{d})$ and $0\leq i<d$
        \[\Hom_A(\trunc^{>-d}(P[i]),\trunc^{>-d}(P'[i]))\cong \Hom_A(P[i],P'[i]) \cong\Hom_A(P,P').\]
        Thus the set 
        \[\mathcal{L}_i\coloneqq\left\{\ [t^{> -d}P[i]]\ \colon  P\in\Proj(\Hc{d}) \text{ is indecomposable}  \right\}\]
        is connected for every $0\leq i<d-1$.
    
        First assume that (\ref{ItemH}) holds. Then, by \cref{TruncIso} it holds for all $1\leq i<d$
        \[\Hom_A(\trunc^{>-d}(A[i]),\trunc^{>-d}(A[i-1]))\cong\Hom_A(A[i],A[i-1])\cong \Hom_A(A,A[-1])\cong H^{-1}(A)\neq 0.\]
        Hence, there exist $P,P'\in\proj(\Hc{d})$ such that 
        \[\Hc{d}(\trunc^{>-d}(P[i]),\trunc^{>-d}(P'[i-1]))\neq 0\]
         which lets us pass between $\mathcal{L}_i$ and $\mathcal{L}_{i-1}$. Hence $A$ is $d$-connected.
        
        Assume now that (\ref{ItemSimple}) holds.  By assumption, there exist simple $H^0(A)$-modules $S,S'$ such that $\Hc{d}(S[i],S'[i+1])\neq 0$ for all $0\leq i<d-1$. For the corresponding indecomposable projectives $P,P'\in\Hc{d}$ it holds
        \[\Hc{d}(\trunc^{>-d}(P[i]),S[i])=\Hom_A(H^0(P),H^0(S))\neq 0, \quad \Hc{d}(\trunc^{>-d}(P'[i+1]),S'[i+1])\neq 0 \]
        which lets us pass between $\mathcal{L}_i$ and $\mathcal{L}_{i+1}$. Hence $A$ is $d$-connected.
    \end{proof}

    \begin{rema}\label{RemarkSemisimple}
        Condition (\ref{ItemSimple}) in the above proposition holds, whenever $H^0(A)$ is connected and not semisimple. Indeed, in that case there exists a non-split extension between simples, as otherwise every $H^0(A)$-module would be a direct sum of simples. However, not every $d$-connected $A$ has to satisfy that $H^0(A)$ is connected, as illustrated in \cref{ExampleH0NotConn}
    \end{rema}

    \section{Explicit Description of the Auslander--Reiten Translation}

    In this section, we discuss the explicit calculation of Auslander--Reiten translations $\tau$ inside of $\Hc{d}$. We first describe the general procedure and illustrate it with an explicit example afterwards. We start with $M\in\Hc{d}$.

    \textbf{First Step}: Compute a $d$-term projective resolution of $M$ as in the proof of \cref{d-res-lemma}, that is a sequence of conflations 
    \[\Omega^iM\stackrel{\iota^i}{\longrightarrow} P^{i-1}\stackrel{p^{i-1}}{\longrightarrow} \Omega^{i-1}M\dashrightarrow \text{ for } 1\leq i\leq d+1,\]
    where $\Omega^0(M)\coloneqq M$ and $P^{i}\in\proj(\Hc{d})$.

    \textbf{Second Step}: Apply the exact functor $\nu_d\coloneqq [d-1]\circ \nu_A$ to obtain
    \[\nu_d(\Omega^iM)\stackrel{\nu_d(\iota^i)}{\longrightarrow} I^{i-1}\stackrel{\nu_d(p^{i-1})}{\longrightarrow} \nu_d(\Omega^{i-1}M)\dashrightarrow \text{ for } 1\leq i\leq d+1,\]
    where $I^i\coloneqq \nu_d(P^{i})\in\inj(\Hc{d})$.

    \textbf{Third Step}: Define $\ker=\trunc^{\leq 0}\circ\Cocone$ and inductively
    \[\Sigma^{1}:=\ker (\nu_d(\iota^1)\circ \nu_d(p^1)), \quad \Sigma^{i+1}:=\ker (I^{i+1}\to \Sigma^{i} ) \text{ for } 1\leq i\leq d-1.\]
    We claim that $\Sigma^d$ computes $\tau M$, that is $\Sigma^d\cong \tau M$.  

    To this end, in the above situation, we define inductively 
    \[I^{[0,1]}:=\Cocone(I^1\to I^0), \quad I^{[0,i]}:=\Cocone(I^{i+1}\to I^{[0,i]}) \text{ for }i\geq 1.\] 
    Recall from the proof of \cref{d-res-lemma}, that $P^{[-i,0]}$ was defined inductively by 
    \[P^{[-1,0]}\coloneqq \cone(P^{-1}\to P^0), \quad P^{[-(i+1),0]}\coloneqq \cone(P^{i+1}[i]\to P^{[-i,0]}) \text{ for }i\geq 1.\]
    Hence $\nu_d(P^{[-i,0]})\cong I^{[0,i]}[-i]$ and in particular, by definition of $\tau$,
    \[\trunc^{\leq 0}I^{[0,d]}\cong\trunc^{\leq 0}\left(\nu_d(P^{[-d,0]})[-d] \right)\cong \trunc^{\leq 0}\left(\nu_A(P^{[-d,0]})[-1]\right)\cong\trunc^{<d}\left(\nu_A(P^{[-d,0]})[-d]\right)[d-1] \cong \tau M.\]  
    
    Thus, it remains to show the following proposition, which is a dual version of \cref{d-res-lemma}.

    \begin{prop}
        There is an isomorphism $\trunc^{\leq{0}}I^{[0,i]}\cong \Sigma^i$ for $1\leq i\leq d$. 
    \end{prop}
    
    \begin{proof}
    We prove the claim by induction on $i$. For $i=1$, the claim follows by definition. For $i\geq 2$, by induction and the octahedral axiom we have the following diagram of triangles:
    $$
    \begin{tikzcd}
    I^i&\Sigma^{{i-1}}& \cone(I^i\to\Sigma^{i-1})&I^i[1]\\
    I^i&I^{[0,i-1]}&I^{[0,i]}[1]&I^i[1]\\
    &\trunc^{\geq 1}I^{[0,i-1]}&\trunc^{\geq 1}I^{[0,i-1]}\\
    & \Sigma^{{i-1}}[1] & \cone(I^i\to\Sigma^{i-1})[1]
    \Ar{1-1}{1-2}{}
    \Ar{1-2}{1-3}{}
    \Ar{1-3}{1-4}{}
    \Ar{2-1}{2-2}{}
    \Ar{2-2}{2-3}{}
    \Ar{2-3}{2-4}{}
    \Ar{3-2}{3-3}{equal}
    \Ar{1-1}{2-1}{equal}
    \Ar{1-4}{2-4}{equal}
    \Ar{1-2}{2-2}{}
    \Ar{1-3}{2-3}{dashed}
    \Ar{2-2}{3-2}{}
    \Ar{2-3}{3-3}{dashed}
    \Ar{3-2}{4-2}{}
    \Ar{3-3}{4-3}{dashed}
    \Ar{4-2}{4-3}{}
    \end{tikzcd}
    $$
    Since $H^{\geq 1}(I^i)=0$, the morphism $H^{\geq 1}(I^{[0,i-1]})\to H^{\geq 1}(I^{[0,i]}[1])$ is an isomorphism. Therefore $\cone(I^{i}\to\Sigma^{i-1})\cong \trunc^{\leq 0}(I^{[0,i]}[1])$ and thus
    \begin{align*}
    \Sigma^i \coloneqq & \trunc^{\leq 0}\Cocone(I^i\to\Sigma^{i-1}) \\
    \cong & \trunc^{\leq 0}\left(\cone\left(I^i\to\Sigma^{i-1}\right)[-1]\right)\\
    \cong & \trunc^{\leq 0}\left(\trunc^{\leq 0}\left(I^{[0,i]}[1]\right)[-1]\right)\\
    \cong & \trunc^{\leq 0}\left(\trunc^{\leq 1}I^{[0,i]}\right) \cong \trunc^{\leq 0}I^{[0,i]}.
    \end{align*}
    \end{proof}

    \begin{ex}\label{ExplicitComputation}
        Let $A$ be the formal dg algebra given by the path algebra of a graded quiver 
        \[A\coloneqq k[{\footnotesize
        \begin{tikzcd}[column sep =13]
            1 & 2
            \Ar{1-1}{1-2}{"\alpha", shift left=1}
            \Ar{1-1}{1-2}{"\beta"', shift right=1,red}
        \end{tikzcd}}],\]
    where $\alpha$ is of degree $0$ and $\beta$ of degree $-1$. As $A$ is concentrated in degrees $-1,0$ we are interested in $\Hc{2}$, that is $d=2$. We have decompositions: 
    \[A=e_2A \oplus e_1A =
    \begin{xsmallmatrix}
    2
    \end{xsmallmatrix}
    \oplus
    \begin{xsmallmatrix}
    &1&\\
    2&&\red{2}
    \end{xsmallmatrix}, 
    \ \nu_2A=\nu_2(Ae_2)\oplus \nu_2(Ae_1)=
    \begin{xsmallmatrix}
    1&&\red{1}\\
    &\red{2}&
    \end{xsmallmatrix}
    \oplus
    \begin{xsmallmatrix}
    \red{1}
    \end{xsmallmatrix}.
    \]
    Here, we denote dg modules by their graded module of cohomologies. By \cite[Thm.~3.6]{KYZ09}, dg modules are determined by their cohomologies up to isomorphism since $A$ is graded hereditary. Red numbers denote degree $-1$ cohomologies and black numbers denote degree $0$ cohomologies. We consider  
    $\begin{xsmallmatrix}
    &1&&\red{1}\\
    2&&\red{2}
    \end{xsmallmatrix}\in\Hc{2}$
    and caluculate 
    $\tau(\begin{xsmallmatrix}
    &1&&\red{1}\\
    2&&\red{2}
    \end{xsmallmatrix})$

    \textbf{First Step}: Calculate a $3$-term projective resolution of 
    $\begin{xsmallmatrix}
        &1&&\red{1}\\
        2&&\red{2}
        \end{xsmallmatrix}$. We have the following conflations: 
    \[
    \begin{xsmallmatrix}
    1
    \end{xsmallmatrix}
    \to
    \begin{xsmallmatrix}
    &1&\\
    2&&\red{2}
    \end{xsmallmatrix}
    \to
    \begin{xsmallmatrix}
        &1&&\red{1}\\
        2&&\red{2}
        \end{xsmallmatrix}
    \dasharrow{},\ 
    \begin{xsmallmatrix}
    2
    \end{xsmallmatrix}
    \oplus
    \begin{xsmallmatrix}
    \red{2}
    \end{xsmallmatrix}
    \to
    \begin{xsmallmatrix}
        &1&\\
        2&&\red{2}
    \end{xsmallmatrix}
    \to
    \begin{xsmallmatrix}
    1
    \end{xsmallmatrix}
    \dasharrow{}
    ,\ 
    \begin{xsmallmatrix}
    2
    \end{xsmallmatrix}
    \xrightarrow{0}
    \begin{xsmallmatrix}
    2
    \end{xsmallmatrix}
    \xrightarrow{{\tiny\begin{bmatrix}1\\0\end{bmatrix}}}
    \begin{xsmallmatrix}
    2
    \end{xsmallmatrix}
    \oplus
    \begin{xsmallmatrix}
    \red{2}
    \end{xsmallmatrix}
    \dasharrow{}
    \]
    
    \leaveout{Thus we get a projective presentaion of $\begin{xsmallmatrix}
        &1&&\red{1}\\
        2&&\red{2}
        \end{xsmallmatrix}$ as follows:
    $$
    \begin{tikzcd}
    {\begin{xsmallmatrix}
    2
    \end{xsmallmatrix}}
    &
    {\begin{xsmallmatrix}
    &1&\\
    2&&\red{2}
    \end{xsmallmatrix}}
    &
    {\begin{xsmallmatrix}
        &1&\\
        2&&\red{2}
        \end{xsmallmatrix}}
    \Ar{1-1}{1-2}{}
    \Ar{1-2}{1-3}{"0"}
    \Ar{1-1}{1-3}{red, bend left=30}
    \end{tikzcd}
    $$}
    \textbf{Second Step}: By applying the exact functor $\nu_2$, we obtain: 
    \[
    \nu_2 (\begin{xsmallmatrix}
    1
    \end{xsmallmatrix})
    \to
    \begin{xsmallmatrix}
    \red{1}
    \end{xsmallmatrix}
    \to
    \nu_2\left(\begin{xsmallmatrix}
        &1&&\red{1}\\
        2&&\red{2}
        \end{xsmallmatrix}\right)
    \dasharrow{},\ 
    \begin{xsmallmatrix}
        1&&\red{1}\\
        &\red{2}&
    \end{xsmallmatrix}
    \oplus
    \begin{xsmallmatrix}
        1&&\red{1}\\
        &\red{2}&
    \end{xsmallmatrix}[1]
    \to
    \begin{xsmallmatrix}
        \red{1}
    \end{xsmallmatrix}
    \to
    \nu_2(\begin{xsmallmatrix}
    1
    \end{xsmallmatrix})
    \dasharrow{}
    ,\ 
    \begin{xsmallmatrix}
        1&&\red{1}\\
        &\red{2}&
    \end{xsmallmatrix}
    \xrightarrow{0}
    \begin{xsmallmatrix}
        1&&\red{1}\\
        &\red{2}&
    \end{xsmallmatrix}
    \xrightarrow{{\tiny\begin{bmatrix}1\\0\end{bmatrix}}}
    \begin{xsmallmatrix}
        1&&\red{1}\\
        &\red{2}&
    \end{xsmallmatrix}
    \oplus
    \begin{xsmallmatrix}
        1&&\red{1}\\
        &\red{2}&
    \end{xsmallmatrix}[1]
    \dasharrow{}
    \]
    
    \textbf{Third Step}: We first observe that the composition $\begin{xsmallmatrix}&1&\\2&&\red{2}\end{xsmallmatrix}\to\begin{xsmallmatrix}1\end{xsmallmatrix}\to \begin{xsmallmatrix}&1&\\2&&\red{2}\end{xsmallmatrix}$ vanishes, hence $\Sigma^1=\trunc^{\leq 0}\Cocone(\begin{xsmallmatrix}
    \red{1}
    \end{xsmallmatrix}\stackrel{0}{\to}\begin{xsmallmatrix}
    \red{1}
    \end{xsmallmatrix})$. The long exact sequence in cohomology yields

    \[0=H^{-2}(\begin{xsmallmatrix}
    \red{1}
    \end{xsmallmatrix})\to H^{-1}(\Sigma^1)\stackrel{\cong}{\to} H^{-1}(\begin{xsmallmatrix}
    \red{1}
    \end{xsmallmatrix}) \stackrel{0}{\to} H^{-1}(\begin{xsmallmatrix}
    \red{1}
    \end{xsmallmatrix})\twoheadrightarrow H^0(\Sigma^1)\to H^{0}(\begin{xsmallmatrix}
    \red{1}
    \end{xsmallmatrix})=0.\]
    Hence $\Sigma^1=\begin{xsmallmatrix}
    1
    \end{xsmallmatrix}\oplus\begin{xsmallmatrix}
    \red{1}
    \end{xsmallmatrix}$. Similarly we get for $\begin{xsmallmatrix}&1&&\red{1}\\&&\red{2}\end{xsmallmatrix}\to \begin{xsmallmatrix}
    1
    \end{xsmallmatrix}\oplus\begin{xsmallmatrix}
    \red{1}
    \end{xsmallmatrix}$ the exact sequences:
    $$
    \begin{tikzcd}[column sep=10]
    0
    &
    H^{-1}(\Sigma^2)
    &
    H^{-1}(\begin{xsmallmatrix}
    &1&&\red{1}\\
    &&\red{2}
    \end{xsmallmatrix}) 
    &
    H^{-1}(\begin{xsmallmatrix}
    1
    \end{xsmallmatrix}\oplus \begin{xsmallmatrix}
    \red{1}
    \end{xsmallmatrix})
    &
    H^0(\Sigma^2)
    &
    H^{0}(\begin{xsmallmatrix}
    &1&&\red{1}\\
    &&\red{2}
    \end{xsmallmatrix}) 
    &
    H^{0}(\begin{xsmallmatrix}
    1
    \end{xsmallmatrix}\oplus\begin{xsmallmatrix}
    \red{1}
    \end{xsmallmatrix})
    \\
    0
    &
    H^{-1}(\Sigma^2)
    &
    \begin{xsmallmatrix}
    1\\2
    \end{xsmallmatrix}
    &
    \begin{smallmatrix}
    1
    \end{smallmatrix}
    &
    H^0(\Sigma^2)
    &
    \begin{xsmallmatrix}
    1
    \end{xsmallmatrix}
    &
    \begin{xsmallmatrix}
    1
    \end{xsmallmatrix}
    \Ar{1-2}{2-2}{equal}
    \Ar{1-3}{2-3}{"\cong"}
    \Ar{1-4}{2-4}{"\cong"}
    \Ar{1-5}{2-5}{equal}
    \Ar{1-6}{2-6}{"\cong"}
    \Ar{1-7}{2-7}{"\cong"}
    \Ar{1-1}{1-2}{}
    \Ar{1-2}{1-3}{}
    \Ar{1-3}{1-4}{}
    \Ar{1-4}{1-5}{}
    \Ar{1-5}{1-6}{}
    \Ar{1-6}{1-7}{}
    \Ar{2-1}{2-2}{}
    \Ar{2-2}{2-3}{hookrightarrow}
    \Ar{2-3}{2-4}{->>}
    \Ar{2-4}{2-5}{"0"}
    \Ar{2-5}{2-6}{"0"}
    \Ar{2-6}{2-7}{"\cong"}
    \end{tikzcd}
    $$
    
    in $\mathrm{mod}\, H^0(A)=\mathrm{mod}\,k[{\footnotesize\begin{tikzcd}[sep = 10]1&2\Ar{1-1}{1-2}{"\alpha"}
    \end{tikzcd}]}$.  Thus we conclude that $\Sigma^2=\tau(\begin{xsmallmatrix}
            &1&&\red{1}\\
            2&&\red{2}
            \end{xsmallmatrix})=\begin{xsmallmatrix}
            \red{2}
            \end{xsmallmatrix}$ 
    \end{ex}

\section{Examples of Auslander--Reiten quivers}

In this section, we assume that $k$ is algebraically closed for simplicity. 

    \begin{ex}
        Let $A$ be the formal dg algebra given by the path algebra of a graded quiver
        \[A\coloneqq k[{\footnotesize
        \begin{tikzcd}[column sep =13]
        1 & 2
        \Ar{1-1}{1-2}{"\alpha", shift left=1}
        \Ar{1-1}{1-2}{"\beta"', shift right=1,red}
        \end{tikzcd}}]\]
        where $\alpha$ is of degree $0$ and $\beta$ of degree $-1$ (see \cref{ExplicitComputation}). Then, the Auslander--Reiten quiver of $\Hc{2}$ looks as follows:
        
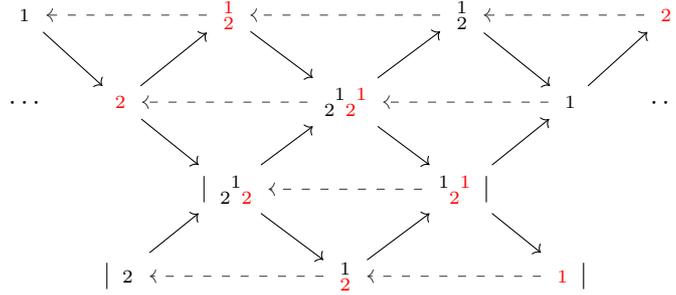
\[\begin{tikzcd}[column sep =10, row sep =10]
        {\begin{xsmallmatrix}
        1
        \end{xsmallmatrix}} 
        &&
        {\begin{xsmallmatrix}
        \red{1}\\\red{2}
        \end{xsmallmatrix}}
        &&
        {\begin{xsmallmatrix}
        1\\2
        \end{xsmallmatrix}}
        &&
        {\begin{xsmallmatrix}
        \red{2}
        \end{xsmallmatrix}}
        \\
        \cdots&
        {\begin{xsmallmatrix}
        \red{2}
        \end{xsmallmatrix}}
        &&
        {\begin{xsmallmatrix}
        &1&&\red{1}\\2&&\red{2}&
        \end{xsmallmatrix}}
        &&
        {\begin{xsmallmatrix}
        1
        \end{xsmallmatrix}}
        &\cdots\\
        &&
        \mid {\begin{xsmallmatrix}
        &1&\\2&&\red{2}
        \end{xsmallmatrix}}
        &&
        {\begin{xsmallmatrix}
        1&&\red{1}\\ &\red{2}&
        \end{xsmallmatrix}} \mid
        &&\\
        &
        \mid {\begin{xsmallmatrix}
        2
        \end{xsmallmatrix}}
        &&
        {\begin{xsmallmatrix}
        1\\\red{2}
        \end{xsmallmatrix}}
        &&
        {\begin{xsmallmatrix}
        \red{1}
        \end{xsmallmatrix}} \mid
        & 
        \Ar{1-1}{2-2}{}
        \Ar{1-3}{2-4}{}
        \Ar{1-5}{2-6}{}
        \Ar{2-2}{1-3}{}
        \Ar{2-4}{1-5}{}
        \Ar{2-6}{1-7}{}
        \Ar{2-2}{3-3}{}
        \Ar{2-4}{3-5}{}
        \Ar{3-3}{2-4}{}
        \Ar{3-5}{2-6}{}
        \Ar{3-3}{4-4}{}
        \Ar{3-5}{4-6}{}
        \Ar{4-2}{3-3}{}
        \Ar{4-4}{3-5}{}
        \tauAr{1-3}{1-1}{}
        \tauAr{1-5}{1-3}{}
        \tauAr{1-7}{1-5}{}
        \tauAr{2-4}{2-2}{}
        \tauAr{2-6}{2-4}{}
        \tauAr{3-5}{3-3}{}
        \tauAr{4-4}{4-2}{}
        \tauAr{4-6}{4-4}{}
        \end{tikzcd}\]

        To prove this we first observe that by \cref{prop: sink of proj}, we have a sink morphism 
        $
        \begin{xsmallmatrix}
        2
        \end{xsmallmatrix}
        \oplus
        \begin{xsmallmatrix}
        \red{2}
        \end{xsmallmatrix}
        \to
        \begin{xsmallmatrix}
        &1&\\2&&\red{2}&
        \end{xsmallmatrix}
        $ and $\begin{xsmallmatrix}2\end{xsmallmatrix}$ has no irreducible morphims from indecomposable objects. Thus, by \cref{prop: iiyatsu}, we can conclude that 
        $
        \begin{xsmallmatrix}
        2
        \end{xsmallmatrix}
        \to
        \begin{xsmallmatrix}
        &1&\\2&&\red{2}&
        \end{xsmallmatrix}
        $
        is a source morphism. As
        $\Coker(\begin{xsmallmatrix}
        2
        \end{xsmallmatrix}
        \to
        \begin{xsmallmatrix}
        &1&\\2&&\red{2}&
        \end{xsmallmatrix})=
        \begin{xsmallmatrix}
        1\\\red{2}
        \end{xsmallmatrix}
        $, we have an almost-split conflation 
        \[\begin{xsmallmatrix}
        2
        \end{xsmallmatrix}
        \to
        \begin{xsmallmatrix}
        &1&\\2&&\red{2}&
        \end{xsmallmatrix}
        \to
        \begin{xsmallmatrix}
        1\\\red{2}
        \end{xsmallmatrix}
        \dashrightarrow.\]
        
        By \cref{ExplicitComputation}, we have $\tau(\begin{xsmallmatrix}&1&&\red{1}\\2&&\red{2}&\end{xsmallmatrix}) \cong \begin{xsmallmatrix}\red{2}\end{xsmallmatrix}$. Thus, we can easily verify the existence of the following almost-split conflation:
        \[
        \begin{xsmallmatrix}
        \red{2}
        \end{xsmallmatrix}
        \to
        \begin{xsmallmatrix}
        \red{1}\\\red{2}
        \end{xsmallmatrix}
        \oplus
        \begin{xsmallmatrix}
        &1&\\2&&\red{2}
        \end{xsmallmatrix}
        \to
        \begin{xsmallmatrix}
        &1&&\red{1}\\2&&\red{2}&
        \end{xsmallmatrix}
        \dashrightarrow
        \]
        
        Since $\begin{xsmallmatrix}\red{2}\end{xsmallmatrix}\to\begin{xsmallmatrix}\red{1}\\\red{2}\end{xsmallmatrix}\oplus\begin{xsmallmatrix}&1&\\2&&\red{2}\end{xsmallmatrix}$
        is a source morphism, 
        $\begin{xsmallmatrix}\red{2}\end{xsmallmatrix}\to\begin{xsmallmatrix}\red{1}\\\red{2}\end{xsmallmatrix}$
        is a sink morphism. Thus, there exists the following almost-split conflation.
        \[ \begin{xsmallmatrix}1\end{xsmallmatrix}\to\begin{xsmallmatrix}\red{2}\end{xsmallmatrix}\to\begin{xsmallmatrix}\red{1}\\\red{2}\end{xsmallmatrix}\dashrightarrow\]
        
        The rest of the Auslander--Reiten quiver can now be determined by knitting.
        
        By the above discussion, we have a connected component of the Auslander--Reiten quiver and since $H^0(A)$ is connected and not semisimple, by \cref{RemarkSemisimple} and \cref{BT1} it is the whole Auslander--Reiten quiver. 

        Similarly, the Auslander--Reiten quiver of $\Hc{3}$ looks as follows, where degree $-2$ cohomologies are depicted in blue:
        $$
        
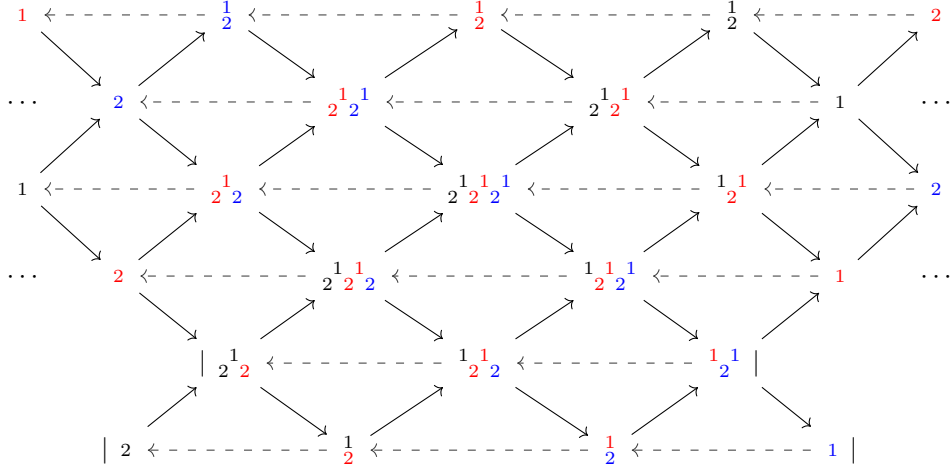
\begin{tikzcd}[column sep =10, row sep =10]
        {\begin{xsmallmatrix}
        \red{1}
        \end{xsmallmatrix}}
        &&
        {\begin{xsmallmatrix}
        \blue{1}\\ \blue{2}
        \end{xsmallmatrix}}
        &&
        {\begin{xsmallmatrix}
        \red{1}\\ \red{2}
        \end{xsmallmatrix}}
        &&
        {\begin{xsmallmatrix}
        1\\2
        \end{xsmallmatrix}}
        &&
        {\begin{xsmallmatrix}
        \red{2}
        \end{xsmallmatrix}}
        \\
        \cdots&
        {\begin{xsmallmatrix}
        \blue{2}
        \end{xsmallmatrix}}
        &&
        {\begin{xsmallmatrix}
        &\red{1}&&\blue{1}\\
        \red{2}&&\blue{2}
        \end{xsmallmatrix}}
        &&
        {\begin{xsmallmatrix}
        &1&&\red{1}\\
        2&&\red{2}
        \end{xsmallmatrix}}
        &&
        {\begin{xsmallmatrix}
        1
        \end{xsmallmatrix}}
        &\cdots\\
        {\begin{xsmallmatrix}
        1
        \end{xsmallmatrix}}
        &&
        {\begin{xsmallmatrix}
        &\red{1}&\\
        \red{2}&&\blue{2}
        \end{xsmallmatrix}}
        &&
        {\begin{xsmallmatrix}
        &1&&\red{1}&&\blue{1}\\
        2&&\red{2}&&\blue{2}&
        \end{xsmallmatrix}}
        &&
        {\begin{xsmallmatrix}
        1&&\red{1}\\
        &\red{2}&
        \end{xsmallmatrix}}
        &&
        {\begin{xsmallmatrix}
        \blue{2}
        \end{xsmallmatrix}}
        \\\cdots&
        {\begin{xsmallmatrix}
        \red{2}
        \end{xsmallmatrix}}
        &&
        {\begin{xsmallmatrix}
        &1&&\red{1}&\\
        2&&\red{2}&&\blue{2}
        \end{xsmallmatrix}}
        &&
        {\begin{xsmallmatrix}
        1&&\red{1}&&\blue{1}\\
        &\red{2}&&\blue{2}&
        \end{xsmallmatrix}}
        &&
        {\begin{xsmallmatrix}
        \red{1}
        \end{xsmallmatrix}}
        &\cdots\\&&
        \mid {\begin{xsmallmatrix}
        &1&\\
        2&&\red{2}
        \end{xsmallmatrix}}
        &&
        {\begin{xsmallmatrix}
        1&&\red{1}&&\\
        &\red{2}&&\blue{2}
        \end{xsmallmatrix}}
        &&
        {\begin{xsmallmatrix}
        \red{1}&&\blue{1}\\
        &\blue{2}&
        \end{xsmallmatrix}} \mid
        &&\\&
        \mid {\begin{xsmallmatrix}
        2
        \end{xsmallmatrix}}
        &&
        {\begin{xsmallmatrix}
        1\\ \red{2}
        \end{xsmallmatrix}}
        &&
        {\begin{xsmallmatrix}
        \red{1}\\ \blue{2}
        \end{xsmallmatrix}}
        &&
        {\begin{xsmallmatrix}
        \blue{1}
        \end{xsmallmatrix}}\mid
        \Ar{1-1}{2-2}{}
        \Ar{1-3}{2-4}{}
        \Ar{1-5}{2-6}{}
        \Ar{1-7}{2-8}{}
        \Ar{2-2}{3-3}{}
        \Ar{2-4}{3-5}{}
        \Ar{2-6}{3-7}{}
        \Ar{2-8}{3-9}{}
        \Ar{3-1}{4-2}{}
        \Ar{3-3}{4-4}{}
        \Ar{3-5}{4-6}{}
        \Ar{3-7}{4-8}{}
        \Ar{4-2}{5-3}{}
        \Ar{4-4}{5-5}{}
        \Ar{4-6}{5-7}{}
        \Ar{5-3}{6-4}{}
        \Ar{5-5}{6-6}{}
        \Ar{5-7}{6-8}{}
        \Arop{1-3}{2-2}{}
        \Arop{1-5}{2-4}{}
        \Arop{1-7}{2-6}{}
        \Arop{1-9}{2-8}{}
        \Arop{2-2}{3-1}{}
        \Arop{2-4}{3-3}{}
        \Arop{2-6}{3-5}{}
        \Arop{2-8}{3-7}{}
        \Arop{3-3}{4-2}{}
        \Arop{3-5}{4-4}{}
        \Arop{3-7}{4-6}{}
        \Arop{3-9}{4-8}{}
        \Arop{4-4}{5-3}{}
        \Arop{4-6}{5-5}{}
        \Arop{4-8}{5-7}{}
        \Arop{5-3}{6-2}{}
        \Arop{5-5}{6-4}{}
        \Arop{5-7}{6-6}{}
        \tauAr{1-3}{1-1}{}
        \tauAr{1-5}{1-3}{}
        \tauAr{1-7}{1-5}{}
        \tauAr{1-9}{1-7}{}
        \tauAr{2-4}{2-2}{}
        \tauAr{2-6}{2-4}{}
        \tauAr{2-8}{2-6}{}
        \tauAr{3-3}{3-1}{}
        \tauAr{3-5}{3-3}{}
        \tauAr{3-7}{3-5}{}
        \tauAr{3-9}{3-7}{}
        \tauAr{4-4}{4-2}{}
        \tauAr{4-6}{4-4}{}
        \tauAr{4-8}{4-6}{}
        \tauAr{5-5}{5-3}{}
        \tauAr{5-7}{5-5}{}
        \tauAr{6-4}{6-2}{}
        \tauAr{6-6}{6-4}{}
        \tauAr{6-8}{6-6}{}
        \end{tikzcd}
        $$

    \end{ex}

    \leaveout{\begin{rema}
        The graded algebra in the above example is derived equivalent to the finite dimensional algebra
        \[\Lambda:=k[{\footnotesize 
        \begin{tikzcd}[column sep=13]
        1&2
        \Ar{1-1}{1-2}{shift left=1,"\alpha"}
        \Ar{1-2}{1-1}{shift left=1,"\beta"}
        \end{tikzcd}}]/(\beta\alpha).\]
        One can prove this by using left mutation of silting objects as defined in \cite[Def.~2.34]{AI}. In their notation, we have $\mu^+_2(A)=\begin{xsmallmatrix}
        &1&\\
        2&&\red{2}
        \end{xsmallmatrix}\oplus\begin{xsmallmatrix}
        1\\ \red{2}
        \end{xsmallmatrix}$ and its dg endomorphism algebra is quasi-isomorphic to $\Lambda$. We can also show the above derived equivalence directly by using mutations of dg quivers as defined in \cite[Section~3]{O17}.
    \end{rema}}

    \begin{ex}
        Let $A$ be the path dg algebra of 
        \[{\small
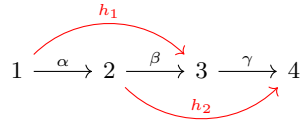
        \begin{tikzcd}
        1&2&3&4
        \Ar{1-1}{1-2}{"\alpha"}
        \Ar{1-2}{1-3}{"\beta"}
        \Ar{1-3}{1-4}{"\gamma"}
        \Ar{1-1}{1-3}{"h_1", bend left=45,red}
        \Ar{1-2}{1-4}{"h_2",bend right=45,red}
        \end{tikzcd}}
        \]
        where black arrows are of degree $0$ and red arrows of degree $-1$, $d(h_1)=\alpha\beta$ and $d(h_2)=\beta\gamma$. This is an honest dg algebra since $[\alpha h_2-h_1 \gamma]\in H^{-1}(A)$ is non-zero and
        \[H^0(A)=k [
        {\small
        \begin{tikzcd}[column sep=13]
        1&2&3&4
        \Ar{1-1}{1-2}{"\alpha"}
        \Ar{1-2}{1-3}{"\beta"}
        \Ar{1-3}{1-4}{"\gamma"}
        \end{tikzcd}}]/(\alpha\beta, \beta\gamma)\]
        $A$ is $d$-connected since $H^0(A)$ is connected and not semisimple.
        \leaveout{This dg algebra is derived equivalent to the finite dimensional algebra $\Lambda$, defined as the path algebra of the following quiver of type $D_4$:
        {\small
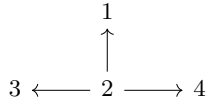
        \[
        \begin{tikzcd}
        & 1 & \\
        3 & 2 & 4
        \Ar{2-2}{2-1}{} 
        \Ar{2-2}{1-2}{} 
        \Ar{2-2}{2-3}{} 
        \end{tikzcd}
        \]}
        This derived equivalence can be checked by realising that $A$ as the dg endomorphism algebra of the silting object $(\mu^+_1 \circ \mu^+_1 \circ \mu^+_2 \circ \mu^+_3)(\Lambda)$ in $\fd(\Lambda)$.}
        Then, one can use the knitting algorithm and \cref{BT1} to determine that the Auslander--Reiten quiver of $\Hc{2}$ looks as follows. 
        
        \[
        \begin{tikzcd}[column sep =10, row sep =10]
            &
            \mid {\begin{smallmatrix}
            3\\4
            \end{smallmatrix}}
            &&
            {\begin{smallmatrix}
            \color{red}{4}
            \end{smallmatrix}}
            &&
            {\begin{smallmatrix}
            2
            \end{smallmatrix}}
            &&
            {\begin{smallmatrix}
            1\\
            \color{red}{3}\\
            \color{red}{4}
            \end{smallmatrix}} \mid
            \\
            \mid {\begin{smallmatrix}
            4
            \end{smallmatrix}}
            &&
            {\begin{smallmatrix}
            3
            \end{smallmatrix}}
            &
            \mid {\begin{smallmatrix}
            2\\3
            \end{smallmatrix}}
            &
            {\begin{smallmatrix}
            2\\
            \color{red}{4}
            \end{smallmatrix}}
            &
            {\begin{smallmatrix}
            \color{red}{3}\\
            \color{red}{4}
            \end{smallmatrix}}
            &
            {\begin{smallmatrix}
            1\\2\\
            \color{red}{3}\\
            \color{red}{4}
            \end{smallmatrix}}
            &
            {\begin{smallmatrix}
            1\\2
            \end{smallmatrix}}
            &
            {\begin{smallmatrix}
            1\\
            \color{red}{3}
            \end{smallmatrix}}
            &
            {\begin{smallmatrix}
            \color{red}{2}\\
            \color{red}{3}
            \end{smallmatrix}} \mid
            &
            {\begin{smallmatrix}
            \color{red}{2}
            \end{smallmatrix}} 
            &&
            {\begin{smallmatrix}
            \color{red}{1}
            \end{smallmatrix}} \mid 
            \\
            &&&&&
            \mid {\begin{smallmatrix}
            1\\
            2\\
            \color{red}{4}
            \end{smallmatrix}}
            &&
            {\begin{smallmatrix}
            \color{red}{3}
            \end{smallmatrix}}
            &&
            {\begin{smallmatrix}
            1
            \end{smallmatrix}}
            &&
            {\begin{smallmatrix}
            \color{red}{1}\\
            \color{red}{2}
            \end{smallmatrix}} \mid 
            \Ar{2-1}{1-2}{}
            \Ar{2-3}{1-4}{}
            \Ar{2-5}{1-6}{}
            \Ar{2-7}{1-8}{}
            \Ar{1-2}{2-3}{}
            \Ar{1-4}{2-5}{}
            \Ar{1-6}{2-7}{}
            \Ar{1-8}{2-9}{}
            \Ar{2-3}{2-4}{}
            \Ar{2-4}{2-5}{}
            \Ar{2-5}{2-6}{}
            \Ar{2-6}{2-7}{}
            \Ar{2-7}{2-8}{}
            \Ar{2-8}{2-9}{}
            \Ar{2-9}{2-10}{}
            \Ar{2-10}{2-11}{}
            \Ar{2-5}{3-6}{}
            \Ar{2-7}{3-8}{}
            \Ar{2-9}{3-10}{}
            \Ar{2-11}{3-12}{}
            \Ar{3-6}{2-7}{}
            \Ar{3-8}{2-9}{}
            \Ar{3-10}{2-11}{}
            \Ar{3-12}{2-13}{}
            \tauAr{1-4}{1-2}{}
            \tauAr{1-6}{1-4}{}
            \tauAr{1-8}{1-6}{}
            \tauAr{2-3}{2-1}{}
            \tauAr{2-5}{2-3}{bend right=40}
            \tauAr{2-7}{2-5}{bend right=40}
            \tauAr{2-9}{2-7}{bend right=40}
            \tauAr{2-11}{2-9}{bend right=40}
            \tauAr{2-13}{2-11}{}
            \tauAr{2-6}{2-4}{bend left=40}
            \tauAr{2-8}{2-6}{bend left=40}
            \tauAr{2-10}{2-8}{bend left=40}
            \tauAr{3-8}{3-6}{}
            \tauAr{3-10}{3-8}{}
            \tauAr{3-12}{3-10}{}
        \end{tikzcd}
        \]
    \end{ex}

    \begin{ex}\label{ExNot2Conn}
    Let $A$ be the graded path-algebra of 
    \[{\small
        \begin{tikzcd}
        1&2&3
        \Ar{1-1}{1-2}{"\alpha",red}
        \Ar{1-2}{1-3}{"\beta"}
        \end{tikzcd}}
        \]
    where $\alpha$ has degree $-1$. Then $A$ is not $2$-connected. The indecomposable projective objects and their truncated shifts are $
    \begin{smallmatrix}
    1\\ \red{2}\\ \red{3}
    \end{smallmatrix},
    \begin{smallmatrix}
    2\\3
    \end{smallmatrix},
    \begin{smallmatrix}
    3\\
    \end{smallmatrix},
    \begin{smallmatrix}
    \red{1}
    \end{smallmatrix},
    \begin{smallmatrix}
    \red{2}\\ \red{3}
    \end{smallmatrix},
    \begin{smallmatrix}
    \red{3}
    \end{smallmatrix}
    $. Then there is no morphism starting or ending in $\begin{smallmatrix}\red{1}\end{smallmatrix}$. 

    The Auslander--Reiten quiver of $\Hc{d}$ is as follows:
    $$
    \begin{tikzcd}[column sep =10, row sep =10]
    &\mid
    {\begin{smallmatrix}
    2\\3
    \end{smallmatrix}}
    &&
    {\begin{smallmatrix}
    \red{3}
    \end{smallmatrix}}
    &&
    {\begin{smallmatrix}
    \red{2}
    \end{smallmatrix}}
    &&
    {\begin{smallmatrix}
    1
    \end{smallmatrix}}
    &&
    {\begin{smallmatrix}
    \red{1}
    \end{smallmatrix}}
    \mid\\\mid
    {\begin{smallmatrix}
    3
    \end{smallmatrix}}
    &&
    {\begin{smallmatrix}
    2
    \end{smallmatrix}}
    &&
    {\begin{smallmatrix}
    \red{2}\\ \red{3}
    \end{smallmatrix}}
    &&
    {\begin{smallmatrix}
    1\\ \red{2}
    \end{smallmatrix}}
    \mid\\
    &&&&&\mid
    {\begin{smallmatrix}
    1\\ \red{2} \\ \red{3}
    \end{smallmatrix}}\mid
    \Ar{2-1}{1-2}{}
    \Ar{2-3}{1-4}{}
    \Ar{2-5}{1-6}{}
    \Ar{2-7}{1-8}{}
    \Ar{1-2}{2-3}{}
    \Ar{1-4}{2-5}{}
    \Ar{1-6}{2-7}{}
    \Ar{2-5}{3-6}{}
    \Ar{3-6}{2-7}{}
    \tauAr{2-3}{2-1}{}
    \tauAr{2-5}{2-3}{}
    \tauAr{2-7}{2-5}{}
    \tauAr{1-4}{1-2}{}
    \tauAr{1-6}{1-4}{}
    \tauAr{1-8}{1-6}{}
    \tauAr{1-10}{1-8}{}
    \end{tikzcd}
    $$
    This Auslander--Reiten quiver is not connected.
    \end{ex}

    \begin{ex}\label{ExampleH0NotConn}
    Let $A$ be the graded path algebra modulo relations of
    \[{\small
    \begin{tikzcd}
    &1&\\
    2&&3
    \Ar{2-1}{1-2}{"\alpha",red}
    \Ar{1-2}{2-3}{"\beta",red}
    \Ar{2-3}{2-1}{"\gamma",red}
    \end{tikzcd}}
    \]
    where the degrees of $\alpha, \beta,\gamma$ are $-1$ and the relations are $\alpha\beta\gamma, \beta\gamma\alpha, \gamma\alpha\beta$. Then $A$ is $3$-connected but $H^0(A)$ is not connected. Since $A\in\Hc{3}$, the category $\Hc{3}$ has Auslander--Reiten sequences. One can easily verify that $A$ is a $2$-self-injective dg algebra (i.e.\ $\Hc{3}$ is a Frobenius extriangulated category).

    We determine the Auslander--Reiten quiver of $\Hc{3}$. At first, we have the following Auslander-Reiten sequence by Proposition \ref{proj inj}:
    $$\begin{smallmatrix}
    \red{2}\\ \blue{3}
    \end{smallmatrix}\to
    \begin{smallmatrix}
    1\\ \red{2} \\ \blue{3}
    \end{smallmatrix}\oplus
    \begin{smallmatrix}
    \red{2}
    \end{smallmatrix}\to
    \begin{smallmatrix}
    1\\ \red{2}
    \end{smallmatrix}\dashrightarrow$$
    Next, we determine $\tau(\begin{smallmatrix}\red{2}\end{smallmatrix})$. 
    We have a following $4$-term projective resolution:
    $$
    \begin{smallmatrix}
    2
    \end{smallmatrix}\to
    0\to
    \begin{smallmatrix}
    \red{2}
    \end{smallmatrix}\dashrightarrow, \ 
    \begin{smallmatrix}
    \red{3}\\ \blue{1}
    \end{smallmatrix}\to
    \begin{smallmatrix}
    2\\ \red{3} \\ \blue{1}
    \end{smallmatrix}\to
    \begin{smallmatrix}
    2
    \end{smallmatrix}\dashrightarrow, \ 
    \begin{smallmatrix}
    3 \\ \red{1}
    \end{smallmatrix}\to
    0\to
    \begin{smallmatrix}
    \red{3} \\ \blue{1}
    \end{smallmatrix}\dashrightarrow, \ 
    \begin{smallmatrix}
    \blue{2}
    \end{smallmatrix}\to
    \begin{smallmatrix}
    3\\ \red{1} \\ \blue{2}
    \end{smallmatrix}\to
    \begin{smallmatrix}
    3 \\ \red{1}
    \end{smallmatrix}\dashrightarrow
    $$
    This is illustrated as follows:
    $$
    \begin{tikzcd}[column sep = 8, row sep = 5]
    &&&
    {\begin{smallmatrix}
    \red{3}\\ \blue{1} 
    \end{smallmatrix}}
    &&&\\
    {\begin{smallmatrix}
    3\\ \red{1}\\ \blue{2}
    \end{smallmatrix}}
    &&
    0
    &&
    {\begin{smallmatrix}
    2\\ \red{3} \\ \blue{1}
    \end{smallmatrix}}
    &&
    0
    &&
    {\begin{smallmatrix}
    \red{2}
    \end{smallmatrix}}\\
    &
    {\begin{smallmatrix}
    3\\ \red{1}
    \end{smallmatrix}}
    &&&&
    {\begin{smallmatrix}
    2
    \end{smallmatrix}}
    \Ar{2-1}{2-3}{}
    \Ar{2-3}{2-5}{}
    \Ar{2-5}{2-7}{}
    \Ar{2-7}{2-9}{}
    \Ar{2-3}{1-4}{}
    \Ar{1-4}{2-5}{}
    \Ar{2-1}{3-2}{}
    \Ar{3-2}{2-3}{}
    \Ar{2-5}{3-6}{}
    \Ar{3-6}{2-7}{}
    \end{tikzcd}
    $$
    By applying $\nu_3$ to this diagram we have a sequence 
    $
    \begin{smallmatrix}
    1\\ \red{2} \\ \blue{3}
    \end{smallmatrix}\to
    0\to
    \begin{smallmatrix}
    3\\ \red{1} \\ \blue{2}
    \end{smallmatrix}\to
    0
    $
    and 
    $$\Sigma^1=\begin{smallmatrix}
    3\\ \red{1}\\ \blue{2}
    \end{smallmatrix}, \ \Sigma^2=\begin{smallmatrix}
    1\\ \red{2}
    \end{smallmatrix}, \ \tau(\begin{smallmatrix}
    \red{2}
    \end{smallmatrix})=\Sigma^3=\begin{smallmatrix}
    \blue{3}
    \end{smallmatrix}$$
    Thus we can conclude that the following is an Auslander--Reiten sequence:
    $$\begin{smallmatrix}
    \blue{3}
    \end{smallmatrix}\to
    \begin{smallmatrix}
    \red{2}\\ \blue{3}
    \end{smallmatrix}\to
    \begin{smallmatrix}
    \red{2}
    \end{smallmatrix}\dashrightarrow.$$
    The other Auslander--Reiten sequences can be calculated by similar arguments and knitting. The Auslander--Reiten quiver looks as follows:
    $$
    \begin{tikzcd}[column sep =8, row sep =8]
    &&&&&&&&\mid 
    {\begin{smallmatrix}
    1\\ \red{2} \\ \blue{3}
    \end{smallmatrix}}
    \mid &&&&&&&\\
    \cdots
    &
    {\begin{smallmatrix}
    \blue{1}
    \end{smallmatrix}}
    &&
    {\begin{smallmatrix}
    \red{3}
    \end{smallmatrix}}
    &&
    {\begin{smallmatrix}
    2
    \end{smallmatrix}}
    &&
    {\begin{smallmatrix}
    \red{2} \\ \blue{3}
    \end{smallmatrix}}
    &&
    {\begin{smallmatrix}
    1 \\ \red{2}
    \end{smallmatrix}}
    &&
    {\begin{smallmatrix}
    \blue{2}
    \end{smallmatrix}}
    &&
    {\begin{smallmatrix}
    \red{1}
    \end{smallmatrix}}
    &&
    {\begin{smallmatrix}
    3
    \end{smallmatrix}}
    \\
    {\begin{smallmatrix}
    3
    \end{smallmatrix}}
    &&
    {\begin{smallmatrix}
    \red{3} \\ \blue{1}
    \end{smallmatrix}}
    &&
    {\begin{smallmatrix}
    2 \\ \red{3}
    \end{smallmatrix}}
    &&
    {\begin{smallmatrix}
    \blue{3}
    \end{smallmatrix}}
    &&
    {\begin{smallmatrix}
    \red{2}
    \end{smallmatrix}}
    &&
    {\begin{smallmatrix}
    1
    \end{smallmatrix}}
    &&
    {\begin{smallmatrix}
    \red{1} \\ \blue{2}
    \end{smallmatrix}}
    &&
    {\begin{smallmatrix}
    3 \\ \red{1}
    \end{smallmatrix}}
    &
    \cdots\\
    &&&\mid
    {\begin{smallmatrix}
    2 \\ \red{3} \\ \blue{1} 
    \end{smallmatrix}}
    \mid&&&&&&&&&&\mid
    {\begin{smallmatrix}
    3 \\ \red{1} \\ \blue{2}
    \end{smallmatrix}}
    \mid
    \Ar{3-1}{2-2}{}
    \Ar{3-3}{2-4}{}
    \Ar{3-5}{2-6}{}
    \Ar{3-7}{2-8}{}
    \Ar{3-9}{2-10}{}
    \Ar{3-11}{2-12}{}
    \Ar{3-13}{2-14}{}
    \Ar{3-15}{2-16}{}
    \Ar{2-2}{3-3}{}
    \Ar{2-4}{3-5}{}
    \Ar{2-6}{3-7}{}
    \Ar{2-8}{3-9}{}
    \Ar{2-10}{3-11}{}
    \Ar{2-12}{3-13}{}
    \Ar{2-14}{3-15}{}
    \Ar{2-8}{1-9}{}
    \Ar{1-9}{2-10}{}
    \Ar{3-3}{4-4}{}
    \Ar{4-4}{3-5}{}
    \Ar{3-13}{4-14}{}
    \Ar{4-14}{3-15}{}
    \tauAr{2-4}{2-2}{}
    \tauAr{2-6}{2-4}{}
    \tauAr{2-8}{2-6}{}
    \tauAr{2-10}{2-8}{}
    \tauAr{2-12}{2-10}{}
    \tauAr{2-14}{2-12}{}
    \tauAr{2-16}{2-14}{}
    \tauAr{3-3}{3-1}{}
    \tauAr{3-5}{3-3}{}
    \tauAr{3-7}{3-5}{}
    \tauAr{3-9}{3-7}{}
    \tauAr{3-11}{3-9}{}
    \tauAr{3-13}{3-11}{}
    \tauAr{3-15}{3-13}{}
    \end{tikzcd}
    $$
    
    \end{ex}
	
    \section{Final Remarks}\label{FinalSection}
    We finish the discussion with the following observation, which highlights again the analogy between $K^{[-1,0]}(\mathrm{proj}(\Lambda))$ and $\per(A)^{[-d,0]}$. In this final section we only assume that $A$ is proper, connective and $d\geq 1$. As $\per(A)^{[-d,0]}\subseteq \fd(A)$ is extension-closed, it has the induced structure of an extriangulated category. Recall from \cite[Def.~3.3]{Chen23d} that an extriangulated category is called \definef{$d$-Auslander}, if it has enough projectives, enough injectives, positive global dimension at most $d+1$ and dominant dimension at least $d+1$, see \cite[Def.~3.1,3.2]{Chen23d} for the definitions.
	
    \begin{lemma}
		The extriangulated category $\per(A)^{[-d,0]}$ is $(d-1)$-Auslander.
    \end{lemma} 
	
    \begin{rema}
		In this sense, \cref{d-term-Equiv} can be viewed as a special case of a ``$(d-1)$-analogue'' of \cite[Prop.~3.11, Sec.~3.3.3]{GNP23}.
    \end{rema}
	
    \begin{proof}
		By the same argument as in \cref{EnoughPI}, we know that $\per(A)^{[-d,0]}\subseteq \fd(A)^{\leq 0}$ has enough projectives given by $\add(A)$. We want to argue it has enough injectives as well. We observe that $\add(A[d])\subseteq\inj(\per(A)^{[-d,0]})$ by \cref{ProjDimd}. Moreover, for every $P\in\per(A)^{[-d,0]}$ there exists by definition a triangle
		\[P'\to P \to P^{-d}[d]\to P'[1]\]
		in $\fd(A)$ such that $P'[1]\in\per(A)^{[-d,0]}$ and $P^{-d}\in\add(A)$. Hence, there exists an inflation $P\to P^{-d}[d]$, which proves there are enough injectives. As in \cref{EnoughPI} one proves $\inj(\per(A)^{[-d,0]})=\add(A[d])$.
		
		It follows from \cref{ProjDimd} that $\per(A)^{[-d,0]}$ has positive global dimension at most $d$. It remains to check that the dominant dimension is at least $d$. Let $P\in\proj(\per(A)^{[-d,0]})=\add(A)$. Then a minimal injective coresolution is given by the following sequence of conflations
		\[P\rightarrowtail 0 \twoheadrightarrow P[1]\dashrightarrow; \; P[1]\rightarrowtail 0 \twoheadrightarrow P[2]\dashrightarrow;\; \cdots;\; P[d-1]\rightarrowtail 0 \twoheadrightarrow P[d]\dashrightarrow;\; P[d]\rightarrowtail P[d] \twoheadrightarrow 0\dashrightarrow,\]
		where $0$ is projective-injective, and $P[d]$ is only injective but not projective if $P\neq 0$. It follows that the dominant dimension of $\per(A)^{[-d,0]}$ is at least $d$ and equal to $d$ if $A\not\simeq 0$.
    \end{proof}

    \subsection*{Acknowledgment}
    The first-named author would like to thank Arashi Sakai for the seminar that had a positive influence on this paper. The author is also grateful to Riku Fushimi and Ryu Tomonaga for valuable discussions, and to Prof. Hiroyuki Nakaoka and Prof. Osamu Iyama for several important comments.

    \subsection*{Financial support}
    The second-named author's research was partly supported by the Swedish Research Council (Vetenskapsrådet) Research Project Grant 2022-03748 `Higher structures in higher-dimensional homological algebra.'
	
    \printbibliography
	
\end{document}